\documentclass[a4paper,reqno,10pt]{amsart}

\usepackage[symbol]{footmisc}

\usepackage{amsmath}
\usepackage{amsthm}
\usepackage{amssymb}
\usepackage{amscd} %diagrams

\usepackage{bbm} %bold numbers
\usepackage{mathrsfs} %nice calligraphy

\usepackage{lmodern} 
\usepackage{newtxtext,newtxmath}

\usepackage{manfnt} % The Bourbaki Z = \dbend

\usepackage{graphicx} %uncomment for inclusion of graphic images

\usepackage{comment} %DO NOT comment this line (comment environment)

\usepackage[Lenny]{fncychap} % other style = Rejne

\usepackage{sqrcaps} % square capital fonts

\usepackage{comment}

\usepackage[svgnames]{xcolor}
\usepackage{pdfcolmk}

\usepackage{tikz}
\usepackage{pifont}

\swapnumbers %theorems numbering appears in front of 'Theorem'

\newtheoremstyle{exercise} %for books or class notes
  {3pt} %space above
  {3pt} %space below
  {\small\rmfamily} %body font
  {
} %indent amount(empty=no indent,\parindent=para indent)
  {\rmfamily\scshape} %thm head font
  {.} %punctuation after thm head
  {.5em} %space after thm head: " " = normal interword space;
  {} %thm head spec (can be left empty, meaning `normal')

\newtheoremstyle{newplain}
  {5pt}
  {5pt}
  {\itshape}
  {}
  {\rmfamily\scshape}
  {. ---}
  {.5em}
  {}

\newtheoremstyle{newremark}
  {5pt}
  {5pt}
  {\rmfamily}
  {}
  {\rmfamily\scshape}
  {. ---}
  {.5em}
  {}

\theoremstyle{newplain}
\newtheorem*{Theorem*}{Theorem} %no numbering for Theorem*

\theoremstyle{newplain}
\newtheorem{Theorem}{Theorem}
\newtheorem{Lemma}[Theorem]{Lemma}
\newtheorem{Corollary}[Theorem]{Corollary}
\newtheorem{Proposition}[Theorem]{Proposition}
\newtheorem{Conjecture}[Theorem]{Conjecture}
\newtheorem{Definition}[Theorem]{Definition}

\theoremstyle{newremark}
\newtheorem{Empty}[Theorem]{}

\newtheorem{Claim}[Theorem]{Claim}

\theoremstyle{exercise}

\numberwithin{Theorem}{section}
\numberwithin{Exercise}{section}

\newcommand{\N}{\mathbb{N}} %natural numbers

\newcommand{\Q}{\mathbb{Q}} %rational numbers
\newcommand{\R}{\mathbb{R}} %real numbers

\newcommand{\Rn}{\R^n}
\newcommand{\ind}{\mathbbm{1}} %indicatrix function

\newcommand{\calA}{\mathscr{A}}
\newcommand{\calB}{\mathscr{B}}
\newcommand{\calC}{\mathscr{C}}

\newcommand{\calE}{\mathscr{E}}
\newcommand{\calF}{\mathscr{F}}
\newcommand{\calG}{\mathscr{G}}
\newcommand{\calH}{\mathscr{H}}
\newcommand{\calI}{\mathscr{I}}

\newcommand{\calL}{\mathscr{L}}
\newcommand{\calM}{\mathscr{M}}
\newcommand{\calN}{\mathscr{N}}

\newcommand{\calP}{\mathscr{P}}

\newcommand{\calS}{\mathscr{S}}
\newcommand{\calT}{\mathscr{T}}

\newcommand{\calX}{\mathscr{X}}

\newcommand{\cac}{\boldsymbol{\frak c}} %puissance du continu

\newcommand{\sfA}{\mathsf{A}}

\newcommand{\sfBan}{\mathsf{Ban_1}}
\newcommand{\sfBool}{\mathsf{Bool}}
\newcommand{\sfCBool}{\mathsf{CBool}}

\newcommand{\sfComptot}{\mathsf{Comp_{tot}}}
\newcommand{\sfCompextr}{\mathsf{Comp_{extr}}}

\newcommand{\sfF}{\mathsf{F}}

\newcommand{\sfL}{\mathsf{L}}
\newcommand{\sfLOC}{\mathsf{LOC}}
\newcommand{\sfMSN}{\mathsf{MSN}}

\newcommand{\sfSpec}{\mathsf{Spec}}
\newcommand{\sfnon}{\mathsf{non}}
\newcommand{\sfcov}{\mathsf{cov}}

\newcommand{\balpha}{\boldsymbol{\alpha}}
\newcommand{\bbeta}{\boldsymbol{\beta}}

\newcommand{\bkappa}{\boldsymbol{\kappa}}

\newcommand{\brho}{\boldsymbol{\rho}}

\newcommand{\bA}{\mathbf{A}}
\newcommand{\bB}{\mathbf{B}}

\newcommand{\bG}{\mathbf{G}}
\newcommand{\bL}{\mathbf{L}}

\newcommand{\bU}{\mathbf{U}}

\newcommand{\bc}{\mathbf{c}}

\newcommand{\boldf}{\mathbf{f}}
\newcommand{\bg}{\mathbf{g}}

\DeclareMathOperator{\rmBdry}{\mathrm{Bdry}} %boundary
\DeclareMathOperator{\rmcard}{\mathrm{card}} %cardinal
\DeclareMathOperator{\rmdiam}{\mathrm{diam}} %diameter

\DeclareMathOperator{\rmgraph}{\mathrm{graph}} %graph
\DeclareMathOperator{\rmideal}{\mathrm{ideal}}
\DeclareMathOperator{\rmim}{\mathrm{im}} %image
\DeclareMathOperator{\rmInt}{\mathrm{Int}} %interior
\DeclareMathOperator{\rmLip}{\mathrm{Lip}} %Lipschitz constant

\DeclareMathOperator{\rmspan}{\mathrm{span}} %span
\DeclareMathOperator{\rmSt}{\mathrm{St}}%espace de Stone
\DeclareMathOperator{\rmTan}{\mathrm{Tan}} %tangent space or measure
\newcommand{\hel} {
\hskip2.5pt{\vrule height7pt width.5pt depth0pt}
\hskip-.2pt\vbox{\hrule height.5pt width7pt depth0pt}
\, }

\def\XXint#1#2#3{{%
\setbox0=\hbox{$#1{#2#3}{\int}$}
\vcenter{\hbox{$#2#3$}}\kern-.5\wd0}}

\newcommand{\veps}{\varepsilon}

\newcommand{\vphi}{\varphi}

\newcommand{\oh}{\frac{1}{2}}

\renewcommand{\em}{\bf}

\renewcommand{\leq}{\leqslant}
\renewcommand{\geq}{\geqslant}
\renewcommand{\subset}{\subseteq}
\renewcommand{\supset}{\supseteq}

\newlength{\drop}

\usepackage{trajan}

\begin{document}

%=================
% TITLE AND AUTHOR
%=================

%\titleAT %see above / page de garde d'un livre

\title[Undecidably semilocalizable measure spaces]{Undecidably semilocalizable metric measure spaces}

\def\curraddrname{{\itshape On leave of absence from}}

\author[Th. De Pauw]{Thierry De Pauw}
\address{School of Mathematical Sciences\\
Shanghai Key Laboratory of PMMP\\ 
East China Normal University\\
500 Dongchuang Road\\
Shanghai 200062\\
P.R. of China\\
and NYU-ECNU Institute of Mathematical Sciences at NYU Shanghai\\
3663 Zhongshan Road North\\
Shanghai 200062\\
China}
\curraddr{Universit\'e Paris Diderot\\ 
Sorbonne Universit\'e\\
CNRS\\ 
Institut de Math\'ematiques de Jussieu -- Paris Rive Gauche, IMJ-PRG\\
F-75013, Paris\\
France}
\email{thdepauw@math.ecnu.edu.cn,thierry.de-pauw@imj-prg.fr}



\thanks{The author was partially supported by the Science and Technology Commission of Shanghai (No. 18dz2271000).}

%\date{December 9th, 2018}

%=========
% ABSTRACT
%=========

\begin{abstract}
We characterize measure spaces such that the canonical map $\bL_\infty \to \bL_1^*$ is surjective.
In case of $d$ dimensional Hausdorff measure on a complete separable metric space $X$ we give two equivalent conditions.
One is in terms of the order completeness of a quotient Boolean algebra associated with measurable sets and with locally null sets.
Another one is in terms of the possibility to decompose space in a certain way into sets of nonzero finite measure.
We give examples of $X$ and $d$ so that whether these conditions are met is undecidable in ZFC, including one with $d$ equals the Hausdorff dimension of $X$.
\end{abstract}

\maketitle

%===========
% DEDICATION
%===========

\begin{comment}
\begin{flushright}
\textsf{\textit{Dedicated to: Joseph Plateau, Arno, Hooverphonic, Hercule
Poirot,\\
Kim Clijsters, Pierre Rapsat, and non Jef t'es pas tout seul}}
\end{flushright}
\end{comment}

%==================
% TABLE OF CONTENTS
%==================

\tableofcontents
%\newpage

%==============================
% THE MEAT --- OR SO ONE THINKS
%==============================

\section{Foreword}

Some questions pertaining to the calculus of variations would benefit from a useful description of the dual of the Banach space $\mathbf{BV}(\Rn)$ of functions of bounded variation in the sense of {\sc E. De Giorgi}.
The question occurs as Problem 7.4 in \cite{ARCATA}.
Measures belonging to this dual space have been characterized by {\sc N.G. Meyers} and {\sc W.P. Ziemer} in \cite{MEY.ZIE.77}.
A description of the other members was obtained (in a slightly different context) by {\sc F.J. Almgren} in \cite{ALM.65.CH} under the Continuum Hypothesis and the particular description was proved to be independent of Zermelo-Fraenkel axioms by the present author in \cite{DEP.98}.
Recently, following former work of {\sc R.D. Mauldin}, {\sc N. Fusco} and {\sc D. Spector} have given a more precise description under the Continuum Hypothesis, \cite{FUS.SPE.18}. 
\par 
In \cite{DEP.98} the problem is shown to be related to describing the dual of the Banach space $\bL_1(\Rn,\calH^{n-1})$ where $\calH^{n-1}$ denotes Hausdorff $n-1$ dimensional measure in $\Rn$.
Here we will restrict to the case when $n=2$ and we shall aim for results in $\mathsf{ZFC}$.
The notation $\bL_1(\Rn,\calH^{n-1})$ however is misleading as it assumes the problem to be independent of the underlying $\sigma$-algebra.
As we shall see, this is not the case.
\par 
Let $(X,\calA,\mu)$ be a measure space.
There is a natural linear retraction
\begin{equation}
\label{eq.100}
\Upsilon : \bL_\infty(X,\calA,\mu) \to \bL_1(X,\calA,\mu)^*
\end{equation}
which sends $\bg$ to $\mathbf{f} \mapsto \int_X gf d\mu$ where $g$ and $f$ represent $\bg$ and $\mathbf{f}$ respectively.
In general $\Upsilon$ does not need to be injective or surjective.
It has been understood for a long time that $\Upsilon$ is injective if and only if $(X,\calA,\mu)$ is {\it semifinite}.
This means that each $A \in \calA$ such that $\mu(A)=\infty$ admits a subset $\calA \ni B \subset A$ with $0 < \mu(B) < \infty$.
Of course every $\sigma$-finite measure space is semifinite.
Yet the dependence upon the $\sigma$-algebra under consideration already occurs in the case of interest to us.
The situation is the following.
\begin{enumerate}
\item If $X$ is a complete separable metric space and $0 < d < \infty$ then the measure space $(X,\calB(X),\calH^d)$ is semifinite. Here $\calB(X)$ denotes the $\sigma$-algebra of Borel subsets of $X$ and $\calH^d$ is the $d$ dimensional Hausdorff measure on $X$. In case $X=\Rn$ this was proved by {\sc R.O. Davies}, \cite{DAV.52} and in general by {\sc J. Howroyd}, \cite{HOW.95}.
\item According to {\sc D.H. Fremlin}, \cite[439H]{FREMLIN.IV} the measure space $(\R^2,\calA_{\calH^1},\calH^1)$ is not semifinite, where $\calA_{\calH^1}$ denotes the $\sigma$-algebra consisting of $\calH^1$ measurable subsets of $\R^2$. This is based on the existence of <<large>> universally null subsets of $[0,1]$ established by {\sc E. Grzegorek}, \cite{GRZ.81}. See also the article of {\sc O. Zindulka} \cite{ZIN.12}.
\end{enumerate}
Nonetheless, recalling our work \cite{DEP.98} it is the surjectivity of $\Upsilon$ that is relevant for the existence of a certain integral representation of members of the dual of $\mathbf{BV}(\R^2)$. 
Injectivity pertains to its uniqueness.
\par
Under the assumption that $(X,\calA,\mu)$ is semifinite, a necessary and sufficient condition for the surjectivity of $\Upsilon$ has been known for a long time.
It asks for the quotient Boolean algebra $\calA/\calN_\mu$ to be order complete, where $\calN_\mu = \calA \cap \{ N : \mu(N) = 0 \}$ is the $\sigma$-ideal of $\mu$ null sets.
Semifinite measure spaces with this property are sometimes called {\it Maharam}, \cite[211G]{FREMLIN.II}.
A stronger condition sometimes called {\it decomposable}, generalizes the idea of $\sigma$-finiteness to possibly uncountable decomposition into sets of finite measure, together with a new condition called {\it locally determined} (that measurability be determined by sets of finite measure), see \ref{61} for the definition of locally determined and \cite[211E]{FREMLIN.II} for the definition of decomposable.
If the quotient $\sigma$-algebra $\calA/\calN_\mu$ is not too big then decomposability implies Maharam according to {\sc E.J. McShane}, \cite{MCS.62} but not in general according to {\sc D.H. Fremlin}, \cite[216E]{FREMLIN.II}.
\par 
If $X$ is a Polish space and $\calB(X)$ denotes the $\sigma$-algebra consisting of its Borel subsets, and if the measure space $(X,\calB(X),\mu)$ is decomposable, then it is $\sigma$-finite. 
I learned the <<counting argument>> to prove this from {\sc D.H. Fremlin}, see \ref{55}.
In view of (1) above it shows that $(\R^2,\calB(\R^2),\calH^1)$ is not decomposable.
Since decomposability is stronger in general than the surjectivity of $\Upsilon$, we need to argue a bit more to show that $(\R^2,\calB(\R^2),\calH^1)$ is not Maharam, see below.
This observation calls for developing a criterion for the surjectivity of $\Upsilon$ without assuming that $(X,\calA,\mu)$ be semifinite in the first place.
We do this in Section 4.
Thus regarding the question whether 
\begin{equation*}
\Upsilon : \bL_\infty\left(\R^2,\calA,\calH^1\right) \to \bL_1\left(\R^2,\calA,\calH^1\right)^*
\end{equation*}
is surjective or not, the situation is the following.
\begin{enumerate}
\item[(3)] If $\calA=\calB(\R^2)$ then $\Upsilon$ is not surjective. Since $(\R^2,\calB(\R^2),\calH^1)$ is semifinite according to (1), and not $\sigma$-finite, it is not decomposable, \ref{55}. The argument of {\sc E.J. McShane}, \ref{mcshane} does not show $(\R^2,\calB(\R^2),\calH^1)$ is not Maharam (the reason being that its completion is not locally determined). However we give below a simple argument to the extent that it is not Maharam, based on Fubini's Theorem.
\item[(4)] If $\calA=\calA_{\calH^1}$ then the surjectivity of $\Upsilon$ is undecidable in $\mathsf{ZFC}$. The consistency of its surjectivity is a consequence of the Continuum Hypothesis, \ref{CH.implies.ad}, \ref{ad.implies.semiloc} and \ref{Riesz}. The consistency of it not being surjective was first noted in \cite{DEP.98} although in a slightly different disguise. The idea is explained below.
\end{enumerate}
\par 
The present paper grew out of the attempt to adapt the techniques used to prove (3) and (4) to the case where $\R^2$ is replaced with a small compact subset $X \subset \R^2$ -- as small as it can possibly be, i.e. of Hausdorff dimension 1 (of course not of $\sigma$-finite $\calH^1$ measure, for in that case $(X,\calB(X),\calH^1)$ and $(X,\calA_{\calH^1},\calH^1)$ are both Maharam and $\Upsilon$ is surjective, \ref{sigmafin}).
Why however would the answer depend on the $\sigma$-algebra under consideration?
In order to understand this, let us try to prove that $\Upsilon$ is surjective.
\par 
We know from the classical Riesz' Theorem that $\Upsilon$ is surjective whenever $(X,\calA,\mu)$ is a finite measure space.
This suggests to consider $\calA^f_\mu = \calA \cap \{ A : \mu(A) < \infty \}$ and for each $A \in \calA^f_\mu$ the map
\begin{equation*}
\Upsilon^A : \bL_\infty(A,\calA_A,\mu_A) \to \bL_1(A,\calA_A,\mu_A)^*
\end{equation*}
where $(A,\calA_A,\mu_A)$ is the obvious measure subspace.
Thus $\Upsilon^A$ is an isometric linear isomorphism and given $\alpha \in  \bL_1(X,\calA,\mu)^*$ there exist $g_A \in \bg_A \in \bL_\infty(A,\calA_A,\mu_A)$ such that
\begin{equation*}
(\alpha \circ \iota_A)(\mathbf{f}) = \int_A g_Af d\mu_A
\end{equation*}
whenever $f \in \mathbf{f} \in \bL_1(A,\calA_A,\mu_A)$, where $\iota_A : \bL_1(A,\calA_A,\mu_A) \to \bL_1(X,\calA,\mu)$ is the obvious embedding.
From the $\mu_A$ almost everywhere uniqueness of the Radon-Nikod\'ym derivative $g_A$ we infer that if $A,A' \in \calA^f_\mu$ then $\mu ( A \cap A' \cap \{ g_A \neq g_{A'} \}) = 0 $.
Thus $(g_A)_{A \in \calA^f_\mu}$ is what we call, from now on a {\it compatible family} of locally defined measurable functions and the question is whether it corresponds to a globally defined measurable function, i.e. whether there exists an $\calA$-measurable $g : X \to \R$ such that $\mu(A \cap \{ g \neq g_A \}) = 0$ for every $A \in \calA^f_\mu$.
If such $g$ exists let us call it a {\it gluing} of the compatible family $(g_A)_{A \in \calA^f_\mu}$.
%
%This is reminiscent of, and not entirely unrelated to the sheaf property of the functor $U \mapsto C(U)$ where $U$ is a subset of a topological space, see \ref{category}(Q3).
%
\par 
It turns out to be rather useful to notice that the question whether a gluing exists or not can be asked in a slightly more general setting since it depends on the measure $\mu$ only insofar as its $\mu$ null sets are involved.
Thus a {\it measurable space with negligibles} $(X,\calA,\calN)$ consists of a measurable space $(X,\calA)$ and a $\sigma$-ideal $\calN \subset \calA$.
Given any $\calE \subset \calA$ one can readily define the notion of a compatible family $(g_E)_{E \in \calE}$ by asking that $E \cap E' \cap \{ g_E \neq g_{E'} \} \in \calN$ whenever $E,E' \in \calE$, and by saying that an $\calA$-measurable function $g : X \to \R$ is a gluing of $(g_E)_{E \in \calE}$ provided $E \cap \{ g \neq g_E \} \in  \calN$ for all $E \in \calE$.
One then shows, \ref{gluing} that each compatible family admits a gluing if and only if each $\calE \subset \calA$ admits an $\calN$ essential supremum $A \in \calA$.
This means that 
\begin{enumerate}
\item[(i)] For every $E \in \calE$ one has $E \setminus A \in \calN$;
\item[(ii)] For every $B \in \calA$, if $E \setminus B \in \calN$ whenever $E \in \calE$, then $A \setminus B \in \calN$. 
\end{enumerate}
We say that a measurable space with negligibles is {\it localizable} if it has this property.
\par 
In this paper we characterize those measure spaces such that $\Upsilon$ is surjective, \ref{Riesz}.
To state this we first define
\begin{equation*}
\calN_\mu\left[\calA^f_\mu\right] = \calA \cap \left\{ N : \mu(A \cap N) = 0 \text{ for all } A \in \calA^f_\mu \right\} \,.
\end{equation*}
It is a $\sigma$-ideal, whose members one is tempted to call {\it locally $\mu$ null}.
\begin{Theorem*}
For any measure space $(X,\calA,\mu)$, the map $\Upsilon$ (recall \eqref{eq.100}) is surjective if and only if the measurable space with negligibles $\left(X,\calA,\calN_\mu\left[\calA^f_\mu\right] \right)$ is localizable.
\end{Theorem*}
We call a measure space {\it semilocalizable} if it has this property -- thus no semifiniteness is assumed.
We study the connection with the notion of {\it almost decomposable} measure space introduced in \cite{DEP.98}, \ref{ad.implies.semiloc} and \ref{mcshane} thereby generalizing to non semifinite measure spaces the classical theory briefly evoked above.
We call a measure space $(X,\calA,\mu)$ {\it almost decomposable} if there exists a disjointed family $\calG \subset \calA^f_\mu$ such that
\begin{equation*}
\forall A \in \calP(X) : ( \forall G \in \calG : A \cap G \in \calA) \Rightarrow A \in \calA \,,
\end{equation*}
and
\begin{equation*}
\forall A \in \calA : \mu(A) < \infty \Rightarrow \mu(A) = \sum_{G \in \calG} \mu(A \cap G) \,.
\end{equation*}
Using an idea of {\sc E.J. McShane}, \cite{MCS.62} and the fact that there are not too many equivalence classes of measurable sets with respect to a Borel regular outer measure on a Polish space, \ref{63} we prove the following, \ref{71}.
\begin{Theorem*}
Let $X$ be a complete separable metric space and $0 < d < 1$. For the measure space $(X,\calA_{\calH^d},\calH^d)$ the following are equivalent.
\begin{enumerate}
\item The canonical map $\Upsilon$ is surjective;
\item $(X,\calA_{\calH^d},\calH^d)$ is semilocalizable;
\item $(X,\calA_{\calH^d},\calH^d)$ is almost decomposable.
\end{enumerate}
\end{Theorem*}
\par 
Let us now consider the measure space $(\R^2,\calB(\R^2),\calH^1)$ in view of the notion of semilocalizability.
We know it is not semilocalizable, (3) above, but we promised to show how this is a consequence of Fubini's Theorem.
Define the vertical sections $V_s = \{s\} \times \R$, $s \in \R$, and the horizontal sections $H_t = \R \times \{t\}$, $t \in \R$. 
Assume if possible that $A \in \calB(\R^2)$ is an $\calN_{\calH^1}\left[\calB(\R^2)^f_{\calH^1}\right]$ essential supremum of the family $(V_s)_{s \in \R}$.
It would then readily follow that
\begin{enumerate}
\item[(a)] $\calH^1(V_s \setminus A) = 0$ for every $s \in \R$;
\item[(b)] $\calH^1(H_t \cap A) = 0$ for every $t \in \R$.
\end{enumerate}
Indeed upon noticing that $V_s$ and $H_t$ have $\sigma$-finite $\calH^1$ measure, (a) is a rephrasing of (i) above and (b) follows from (ii) applied with $B=A \setminus H_t$.
Applying Fubini's Theorem twice would yield
\begin{equation*}
\calL^2(\R^2 \setminus A) = \int_\R \calH^1( V_s \setminus A) d\calL^1(s) = 0
\end{equation*}
according to (a), and
\begin{equation*}
\calL^2(\R^2 \cap A) = \int_\R \calH^1( H_t \cap A) d\calL^1(t) = 0
\end{equation*}
according to (b). 
In turn $\calL^2(\R^2)=0$, a contradiction. 
Clearly the same argument applies with $\R^2$ replaced by any Borel set $X \subset \R^2$ such that $\calL^2(X) > 0$, to showing that $(X,\calB(X),\calH^1)$ is not semilocalizable.
\par 
There are two cases when the above argument is not conclusive: 
\begin{enumerate}
\item[($\alpha$)] when $A$ is not $\calL^2$ measurable (because Fubini's Theorem does not apply);
\item[($\beta$)] when $\calL^2(X)=0$ (because no contradiction ensues).
\end{enumerate}
\par 
With regard to case ($\alpha$) indeed, when we replace the $\sigma$-algebra $\calB(\R^2)$ by the larger $\calA_{\calH^1}$ then $(\R^2,\calA_{\calH^1},\calH^1)$ is consistently semilocalizable.
This is a consequence of the Continuum Hypothesis and a much more general statement holds\footnote{I learned it from \cite[2.5.10]{GMT}. Unfortunately the presentation there does not allow for putting emphasis on the role played by the choice of a particular $\sigma$-algebra.}, \ref{CH.implies.ad}.
As noticed in \cite{DEP.98} it turns out however that $(\R^2 , \calA_{\calH^1}, \calH^1)$ is {\it also} consistently {\it not} semilocalizable.
Here is the reason why.
We assume that $A \in \calA_{\calH^1}$ is an $\calN_{\calH^1}\left[\calA_{\calH^1}^f\right]$ essential supremum of the family $(V_s)_{s \in \R}$.
For each $s \in \R$ we define $T_s = \R \cap \{ t : (s,t) \in V_s \setminus A \}$, thus $\calL^1(T_s)=0$ according to (a).
Now choose $E \subset \R$ such that $\calL^1(E) > 0$ and $E$ has least cardinal among all sets with nonzero Lebesgue measure, and let $\sfnon(\calN_{\calL^1})$ denote this cardinal.
Assume that there exists $t \in \R \setminus \cup_{s \in E} T_s$.
Then for each $s \in E$, $t \not \in T_s$, i.e. $(s,t) \in H_t \cap A$.
Therefore $\calL^1(E) = 0$ according to (b), a contradiction.
Of course we can reach this contradiction only if $\R \neq \cup_{s \in E} T_s$, which depends upon how big $E$ is.
We denote as $\sfcov(\calN_{\calL^1})$ the least cardinal of a covering of $\R$ by $\calL^1$ negligible sets. 
Thus if $\rmcard E = \sfnon(\calN_{\calL^1}) < \sfcov(\calN_{\calL^1})$ then the argument goes through.
It turns out that this strict inequality of cardinals (appearing in the so-called Cicho\'n diagram) is consistent with $\mathsf{ZFC}$, \cite[Chapter 7]{BARTOSZYNSKI.JUDAH} or \cite[552H and 552G]{FREMLIN.V.2}.
We will refer to this idea below as the <<vertical-horizontal method>>.
This argument is from \cite{DEP.98} ; I learned it from {\sc D.H. Fremlin}.
\par 
With regard to case ($\beta$) above we observe again that the Continuum Hypothesis implies that $(X,\calA_{\calH^1},\calH^1)$ is semilocalizable for any compact set $X \subset \R^2$ regardless whether it has zero $\calL^2$ measure or not, \ref{CH.implies.ad}.
The question is therefore whether $(X,\calA_{\calH^1},\calH^1)$ is semilocalizable in $\mathsf{ZFC}$ or consistently not semilocalizable.
The latter occurs when the <<vertical-horizontal method>> generalizes from $X = \R^2$ to $X$.
For instance it clearly generalizes to $X = [a,b ] \times [c,d]$ but it is not instantly obvious how to proceed if $\calL^2(X) = 0$.
Thus we ought to explain how the <<vertical-horizontal method>> described above, showing that if $\sfnon(\calN_{\calL^1}) < \sfcov(\calN_{\calL^1})$ then $\left(\R^2,\calA_{\calH^1},\calH^1\right)$ is not semilocalizable, can be adapted to the case where $\R^2$ is replaced with some suitable subset $X \subset \R^2$.
We give a rather general version below, \ref{abstract.theorem}.
First of all we make the useful observation that if $(S,\calB(S),\sigma)$ is a probability space, $S$ is Polish and $\sigma$ is diffuse, then $\sfnon(\calN_{\bar{\sigma}}) = \sfnon(\calN_{\calL^1})$ and $\sfcov(\calN_{\bar{\sigma}}) = \sfcov(\calN_{\calL^1})$.
This ensues from the Kuratowski Isomorphism Theorem, \ref{number.8}.
A careful inspection of the argument leads to the following, \ref{85}.
\begin{Theorem*}
Let $0 < d < 1$ and let $C_d \subset [0,1]$ be the standard self-similar Cantor set of Hausdorff dimension $0 < d < 1$.
Whether the measure space $(C_d \times C_d , \calA_{\calH^d}, \calH^d)$ is semilocalizable is undecidable in $\mathsf{ZFC}$.
\end{Theorem*}
Incidentally, constructing a certain isomorphism in the category of measurable spaces with negligibles we are able to infer the following, \ref{pue.8}.
\begin{Theorem*}
Whether the measure space $\left([0,1],\calA_{\calH^{1/2}},\calH^\oh\right)$ is semilocalizable is undecidable in $\mathsf{ZFC}$.
\end{Theorem*}
Here the exponent $1/2$ reflects the nature of the argument, viewing the space $X$ as a product of a kind, where <<vertical>> sets $V_s$ and <<horizontal>> sets $H_t$ of the same size make sense, their intersections behaving according to some technical assumptions (see the statement of \ref{abstract.theorem}).
Note that the sets $X = C_d \times C_d$ are purely $(\calH^1,1)$ unrectifiable, \ref{pue.1}(1).
% and therefore irrelevant to the question whether the existence of a weak tangent field for $X$ follows from a localizability property: Any map $X \to \bG(\R^2,1)$ is a weak tangent field if $X$ is purely $(\calH^1,1)$ unrectifiable, let alone the fact that $d=1$ is omitted above.
%
%\par 
%Therefore 
We next seek to apply the <<vertical-horizontal method>> to an $\calL^2$ negligible compact set $X \subset \R^2$ which is not purely $(\calH^1,1)$ unrectifiable and prove that $(X,\calA_{\calH^1},\calH^1)$ is not semilocalizable.
Let us choose $X$ as small as possible, i.e. of Hausdorff dimension 1, say $X = C \times [0,1]$ where $C \subset [0,1]$ is a Cantor set of Hausdorff dimension 0.
It is of course clear that $V_s = \{s\} \times [0,1]$, $s \in C$, can be chosen as our vertical sets, yet the choice $H_t = C \times \{t\}$, $t \in [0,1]$, will be of no use since $\calH^1(H_t)=0$ and therefore no contradiction can ensue when implementing the <<vertical-horizontal method>>.
Instead we proceed as follows to define $H_t$.
Let $\mu$ be a diffuse probability measure on $C$ and let $f(t) = \mu([0,t])$, $t \in [0,1]$, be its distribution function (this is a version of the Cantor-Vitali devil staircase for our 0 dimensional set $C$).
Consider the graph $G$ of the function $\frac{1}{2}f$ ; thus $G$ is a rectifiable curve, and intersects non $\calH^1$ trivially the set $X$, \ref{number.6}.
We then define $H_t = G + t.e_2$, $t \in [0,1/2]$, where $e_2=(0,1)$. 
It turns out that these will successfully play the role of horizontal sets, the details are in section \ref{example}.
The following subsumes \ref{number.9} and \ref{CR.1}.
\begin{Theorem*}
Assume that
\begin{enumerate}
\item $C \subset [0,1]$ is some Cantor set of Hausdorff dimension 0; 
\item $X = C \times [0,1]$;
\item $\calA$ is a $\sigma$-algebra and $\calB(X) \subset \calA \subset \calP(X)$;
\item $\calN = \calN_{\calH^1}$ or $\calN = \calN_{pu}$. 
\end{enumerate}
It follows that the measurable space with negligibles $(X,\calA,\calN)$ is consistently not localizable.
\end{Theorem*}
%
\begin{comment}
Of course this particular set $X$ admits an obvious constant weak tangent field: $\tau(x) = \rmspan\{e_2\}$, $x \in X$.
%
Notice in particular that for $\calH^1$ almost every $x \in G \cap X$ the tangent line to $G$ at $x$ is vertical.
%
This illustrates that localizability is indeed much stronger than the existence of a weak tangent field. 
\end{comment}
\par 

My thanks are due to {\sc David H. Fremlin}, not only for his inspiring treatise <<Measure Theory>> \cite{FREMLIN.I,FREMLIN.II,FREMLIN.III,FREMLIN.IV,FREMLIN.V.1,FREMLIN.V.2} but also for many helpful conversations.
%
%I am also indebted to {\sc Francis Borceux} for useful conversations regarding \ref{category}.
%
Finally I am also grateful to {\sc ZhiQiang Wang} for his careful reading and witty comments.

\section{Preliminaries}

\begin{Empty}
A {\em measurable space} consists of a pair $(X,\calA)$ where $X$ is a set and $\calA$ is a $\sigma$-algebra of subsets of $X$. 
Whenever $(X,\calA)$ and $(Y,\calB)$ are measurable spaces and $f : X \to Y$ we say that $f$ is {\em $(\calA,\calB)$ measurable} provided $f^{-1}(B) \in \calA$ for all $A \in \calA$.
In the particular case when $Y=\R$ it is always understood that $\calB = \calB(\R)$ is the $\sigma$-algebra consisting of Borel subsets of $\R$ and we say that $f$ is {\em $\calA$ measurable} instead of $(\calA,\calB(\R))$ measurable.
We let $L_0(X,\calA)$ denote the collection of $\calA$ measurable functions $X \to \R$.
It is an algebra and a Riesz space (under the pointwise operations and partial order).
\end{Empty}

\begin{Empty}
As usual a {\em measure space} $(X,\calA,\mu)$ consists in a measurable space $(X,\calA)$ and a measure $\mu$ defined on the $\sigma$-algebra.
We let $L_1(X,\calA,\mu)$ denote the subspace of $L_0(X,\calA)$ consisting of those $f$ such that $|f|$ is $\mu$-summable.
The corresponding space of equivalence classes with respect to equality $\mu$ almost everywhere is denoted $\bL_1(X,\calA,\mu)$.
If $f \in L_1(X,\calA,\mu)$ we let $f^\bullet$ denote its equivalence class in $\bL_1(X,\calA,\mu)$.
Thus $f \in \mathbf{f} \in \bL_1(X,\calA,\mu)$ means that $f$ is an actual function representing the equivalence class $\mathbf{f}$, i.e. $f^\bullet = \mathbf{f}$.
\end{Empty}

\begin{Empty}
If $(X,\calA,\mu)$ is a measure space and $\calB \subset \calA$ is a $\sigma$-algebra, we let $\mu|_\calB$ denote the restriction of $\mu$ to $\calB$.
If $A \in \calA$ we let $(A,\calA_A,\mu_A)$ denote the measure space where $\calA_A = \calB \cap \{ B : B \subset A \}$ and $\mu_A = \mu|_{\calA_A}$.
\end{Empty}

\begin{Empty}
\label{P1}
With a measure space $(X,\calA,\mu)$ we associate an outer measure $\bar{\mu}$ on $X$ by the usual formula
\begin{equation*}
\bar{\mu}(S) = \inf \{ \mu(A) : \calA \ni A \supset S \} \,.
\end{equation*} 
\end{Empty}

\begin{Empty}
If $X$ is a set and $\phi$ an outer measure on $X$ we let $\calA_\phi$ denote the $\sigma$-algebra consisting of those subsets of $X$ which are $\phi$ measurable in the sense of Caratheodory. 
\end{Empty}

\begin{Empty}
\label{P3}
{\it 
If $(X,\calA,\mu)$ is a measure space and $\bar{\mu}$ is associated with it as in \ref{P1} then
\begin{enumerate}
\item $\calA \subset \calA_{\bar{\mu}}$;
\item For every $A \in \calA$ one has $\bar{\mu}(A) = \mu(A)$;
\item For every $S \subset X$ there exists $\calA \ni A \supset S$ such that $\bar{\mu}(S) = \mu(A)$.
\end{enumerate}
}
\par 
Let $A \in \calA$ and $S \subset X$. 
We ought to show that $\bar{\mu}(S) \geq \bar{\mu}(S \cap A) + \bar{\mu}(S \setminus A)$.
Let $\calA \ni B \supset S$ and notice that $\mu(B) = \mu(B \cap A) + \mu(B \setminus A) \geq \bar{\mu}(S \cap A) + \bar{\mu}(S \setminus A)$.
Since $B$ is arbitrary the proof of (1) is complete.
Given $A \in \calA$ and $\calA \ni B \supset A$ we clearly have $\mu(A) \leq \mu(B)$ and, since $B$ is arbitrary $\mu(A) \leq \bar{\mu}(A)$.
Letting $B=A$ proves the equality of conclusion (2) and we now turn to establishing (3).
If $\bar{\mu}(S)=\infty$ then take $A = X$.
If not choose $\calA \ni A'_n \supset S$ such that $\mu(A'_n) \leq n^{-1} + \bar{\mu}(S)$, $n \in \N^*$, let $A_n = \cap_{m \geq n} A'_m$, $A = \cap_{n \in \N^*} A_n$ and notice that $\calA \ni A \supset S$ and $\bar{\mu}(S) \leq \mu(A) = \lim_n \mu(A_n) \leq \lim_n \mu(A'_n) = \bar{\mu}(S)$.
\end{Empty}

\begin{Empty}
\label{24}
If $X$ is a Polish space and $\phi$ an outer measure on $X$ we say that $\phi$ is {\em Borel regular} if
\begin{enumerate}
\item $\calB(X) \subset \calA_\phi$, i.e. each Borel subset of $X$ is $\phi$ measurable;
\item For every $A \subset X$ there exists a Borel set $B \subset X$ such that $A \subset B$ and $\phi(A)=\phi(B)$.
\end{enumerate}
When $B$ is associated with $A$ as in (2) we call it a {\em Borel hull} of $A$.
In this case one readily checks that $\phi(B) \leq \phi(B')$ whenever $B' \in \calB(X)$ and $A \subset B'$.
\end{Empty}

\begin{Empty}
{\it 
If $X$ is a Polish space, $\phi$ is a Borel regular outer measure on $X$, and $\mu = \phi|_{\calB(X)}$, then $\bar{\mu} = \phi$.
}
\par 
This is a particular case of \ref{P3}(2).
\end{Empty}

\begin{Empty}
\label{26}
{\it 
If $(X,\calB(X),\mu)$ is a measure space where $X$ is Polish then $\bar{\mu}$ is Borel regular and $\bar{\mu}|_{\calB(X)} = \mu$.
}
\par 
This is a particular case of \ref{P3}.
\end{Empty}

\begin{Empty}
Let $X$ be a metric space and $0 < d < \infty$.
Given $0 < \delta \leq \infty$ and $A \subset X$ we define
\begin{equation*}
\calH^d_{(\delta)}(A) = \inf \left\{ \sum_{i \in I} (\rmdiam A_i)^d : A \subset \cup_{i \in I} A_i, I \text{ is at most countable, and } \rmdiam A_i \leq \delta \right\} \,.
\end{equation*}
We further let
\begin{equation*}
\calH^d(A) = \lim_{\delta \to 0^+} \calH^d_{(\delta)}(A) = \inf \left\{ \calH^d_{(\delta)}(A) : 0 < \delta \leq \infty\right\} \,.
\end{equation*}
Thus $\calH^d$ is a Borel regular outer measure on $X$.
Notice that our definition differs from that of \cite[2.10.2]{GMT} by a constant multiplicative factor.
This does not affect the results stated in the present paper, except for the specific constants in \ref{pue.6} which are of no relevance otherwise to our concerns.
\end{Empty}

\section{Measurable Spaces with Negligibles}

Most of the material in this Section is either known, or folklore or both, with the possible exception of the Definitions and Facts in \ref{localized.negligibles} and \ref{ideals} needed in the next Section.
I learned about the concept of measurable space with negligibles in \textsc{D.H. Fremlin}'s treatise on Measure Theory.
Here I call localizable a measurable space with negligibles whose quotient Boolean algebra is order complete, an important class of examples being the $\sigma$-finite measures spaces, \ref{sigmafinite.localizable}.
The main property of localizable measurable spaces with negligibles needed in the remaining part of this paper is the possibility of gluing, in a globally measurable way, the locally almost everywhere  defined measurable functions, \ref{gluing}.
We spell out the proof which is similar to the case of measure spaces, see \cite[213N]{FREMLIN.II} for one direction.
%
%We close the Section with categorical remarks and suggestions for future work.

\begin{Empty}
\label{def.MSN}
A {\em measurable space with negligibles} consists of a triple $(X,\calA,\calN)$ where $(X,\calA)$ is a measurable space and $\calN \subset \calA$ is a $\sigma$-ideal of $\calA$.
The latter means that:
\begin{enumerate}
\item $\emptyset \in \calN$;
\item If $A \in \calA$, $B \in \calN$ and $A \subset B$ then $A \in \calN$;
\item If $(A_n)_{n \in \N}$ is a sequence in $\calN$ then $\cup_{n \in \N} A_n \in \calN$.
\end{enumerate}
Given a measure space $(X,\calA,\mu)$ we define
\begin{equation*}
\calN_\mu = \calA \cap \{ N : \mu(N) = 0 \}
\end{equation*}
so that clearly $(X,\calA,\calN_\mu)$ is a measurable space with negligibles.
Even though it seems the natural measurable space with negligibles associated with $(X,\calA,\mu)$, it is by no means the only one that will matter in this paper, see \ref{def.semilocalizable}. 
Similarly if $\phi$ is an outer measure on a set $X$ then
\begin{equation*}
\calN_\phi = \calP(X) \cap \{ N : \phi(N) = 0 \}
\end{equation*}
is a $\sigma$-ideal of $\calP(X)$.
\end{Empty}

\begin{Empty}
Given a measurable space with negligibles $(X,\calA,\calN)$ and $g \in L_0(X,\calA)$ we define
\begin{equation*}
\| g \|_\calN = \inf \{ t : X \cap \{ x : |g(x)| > t \} \in \calN \} \in [0,\infty]
\end{equation*}
and we say that $g$ is {\em $\calN$ essentially bounded} if $\|g\|_\calN < \infty$.
Letting $L_\infty(X,\calA,\calN)$ denote the collection of such functions and be equipped with the operations and partial order inherited from $L_0(X,\calA)$ one checks it is an algebra and a Riesz space.
Furthermore $\|\cdot\|_\calN$ is a seminorm defined on $L_\infty(X,\calA,\calN)$. 
One classically shows that $\|g\|_\calN = 0$ if and only $X \cap \{ g \neq 0 \} \in \calN$ and we let $\bL_\infty(X,\calA,\calN)$ be the corresponding quotient space equipped with the corresponding norm.
The following is established in exactly the same way as in the case of measure spaces.
\end{Empty}

\begin{Empty}
{\it 
Given a measurable space with negligibles $(X,\calA,\calN)$, $\bL_\infty(X,\calA,\calN)$ is both a Banach space and a Banach lattice.
}
\end{Empty}

\begin{Empty}
\label{def.stone}
Let $(X, \calA,\calN)$ be a measurable space with negligibles. 
Forgetting about the stability of $\calA$ and $\calN$ under {\it countable (rather than finite)} operations we view $\calA$ as a Boolean algebra and $\calN$ as an ideal of $\calA$. 
As such the quotient $\calA_\calN := \calA / \calN$ is a Boolean algebra as well.
\par 
Given an arbitrary Boolean algebra $\bB$ we recall that its {\em Stone representation} $\sfSpec(\bB)$ is a totally disconnected compact Hausdorff topological space of which the Boolean algebra of clopen sets is isomorphic to $\bB$, see e.g. \cite[311E and 311I]{FREMLIN.III}.
By a totally disconnected topological space we mean one whose connected subsets are all singletons ; if the space is assumed to be compact Hausdorff this is equivalent to the existence of a basis for the topology consisting of clopen (closed and open) subsets.
\end{Empty}

\begin{Proposition}
\label{LC}
Given a measurable space with negligibles $(X,\calA,\calN)$, the Banach spaces $\bL_\infty(X,\calA,\calN)$ and $C(\sfSpec(\calA_\calN))$ are isometrically isomorphic.
\end{Proposition}

\begin{proof}
Letting $\bL_{\infty,s}(X,\calA,\calN)$ denote the linear subspace of $\bL_\infty(X,\calA,\calN)$ corresponding to those simple functions $g \in L_\infty(X,\calA,\calN)$, i.e. those having finite range, we define $\Xi : \bL_{\infty,s}(X,\calA,\calN) \to C(\sfSpec(\calA_\calN))$ by the formula
\begin{equation*}
\Xi(u^\bullet) = \sum_{y \in u(X)} y \ind_{\rmSt(u^{-1}\{y\}^\bullet)}
\end{equation*}
where $\rmSt : \calA_\calN \to \sfSpec(\calA_\calN)$ is the Stone isomorphism and the superscript bullet denotes the equivalence class.
Since each $\ind_{\rmSt(A^\bullet)}$, $A \in \calA$, is continuous, $\Xi$ is well defined.
It is easy to check that $\Xi$ is a linear isometry onto its image.
The basic Approximation Lemma of measurable functions by simple functions implies that $\bL_{\infty,s}(X,\calA,\calN)$ is dense in $\bL_\infty(X,\calA,\calN)$, therefore $\Xi$ uniquely extends to a linear isometry $\hat{\Xi} : \bL_\infty(X,\calA,\calN) \to C(\sfSpec(\calA_\calN))$.
Upon noticing that $\rmim \Xi$ is a subalgebra of $C(\sfSpec(\calA_\calN))$ that contains the constant functions and that separates points we infer from the Stone-Weierstrass Theorem that $\rmim \Xi$ is dense and in turn that $\hat{\Xi}$ is surjective. 
\end{proof}

\begin{Empty}
\label{def.ess.sup}
Let $(X,\calA,\calN)$ be a measurable space with negligibles and $\calE \subset \calA$. We say that $A \in \calA$ is an {\em $\calN$ essential supremum} of $\calE$ whenever the following holds:
\begin{enumerate}
\item For every $E \in \calE$ one has $E \setminus A \in \calN$;
\item If $B \in \calA$ is such that $E \setminus B \in \calN$ for every $E \in \calE$, then $A \setminus B \in \calN$.
\end{enumerate}
In particular if $A,A' \in \calA$ are both $\calN$ essential suprema of $\calE$ it follows that $A \ominus A' \in \calN$ where $\ominus$ denotes the symmetric difference of two sets.
If $A$ verifies condition (1) but necessarily condition (2) we call it an {\em $\calN$ essential upper bound} of $\calE$.
\end{Empty}

\begin{Empty}
\label{abstract.loc}
We say that a measurable space with negligibles $(X,\calA,\calN)$ is {\em localizable} whenever each family $\calE \subset \calA$ admits an $\calN$-essential supremum.
{\it 
Given a measurable space with negligibles $(X,\calA,\calN)$ the following conditions are equivalent:
\begin{enumerate}
\item $(X,\calA,\calN)$ is localizable;
\item The Boolean algebra $\calA_\calN$ is order complete;
\item The Stone space $\sfSpec(\calA_\calN$) is extremally disconnected;
\item The Banach lattice $C(\sfSpec(\calA_\calN))$ is order complete;
\item The Banach space $C(\sfSpec(\calA_\calN))$ is isometrically injective.
\end{enumerate}
}
That (1) be equivalent to (2) is routine verification.
If $\bB$ is a Boolean algebra then $\bB$ is order complete if and only if $\sfSpec(\bB)$ is extremally disconnected, see e.g. \cite[314S]{FREMLIN.III}.
We recall that a compact Hausdorff topological space $K$ is called extremally disconnected if the closure of any open set is open.
Furthermore a compact Hausdorff space $K$ is extremally disconnected if and only if $C(K)$ is order complete, see e.g. \cite[Problems 4.5 and 4.6]{ALBIAC.KALTON}.
This shows the equivalence between (3) and (4).
The equivalence between (4) and (5) is a consequence of the Goodner-Nachbin Theorem, see \cite[4.3.6]{ALBIAC.KALTON}.
\par 
An important class of examples of localizable spaces with negligibles is given below, with a proof for the reader's convenience.
\end{Empty}

\begin{Proposition}
\label{sigmafinite.localizable}
If $(X,\calA,\mu)$ is a $\sigma$-finite measure space then $(X,\calA,\calN_\mu)$ is localizable (recall \ref{def.MSN}).
\end{Proposition}

\begin{proof}
If $(X,\calA,\mu)$ is not finite choose a partition $(X_n)_{n \in \N}$ of $X$ into members of $\calA$ such that $0 < \mu(X_n) < \infty$ for every $n \in \N$ and define a measure $\nu$ on $\calA$ by the formula
\begin{equation*}
\nu(A) = \sum_{n \in \N} 2^{-n} \mu(X_n)^{-1} \mu(X_n \cap A) \,,
\end{equation*}
$A \in \calA$. 
Observing that $\calN_\nu = \calN_\mu$ and that $\nu(X)=2$ we conclude that the proposition follows from its special case when $(X,\calA,\mu)$ is finite.
\par 
We henceforth assume that $\mu(X) < \infty$.
Let $\calE \subset \calA$ and define 
\begin{equation*}
\calF = \calA \cap \{ F : \mu(F \cap E)=0 \text{ for every } E \in \calE \} \,.
\end{equation*}
Notice that $\calF$ is a $\sigma$-ideal. 
Put $\tau = \sup \{ \mu(F) : F \in \calF \} < \infty$.
There exists a nondecreasing sequence $(F_n)_{n \in \N}$ in $\calF$ such that $\mu(F_n) \geq \tau - (n+1)^{-1}$ for every $n \in \N$.
Thus $F := \cup_{n \in \N} F_n \in \calF$ and $\mu(F) = \tau$.
In particular $\mu(G \setminus F) = 0$ for every $G \in \calF$ (for otherwise $\mu(F \cup (G \setminus F)) > \mu(F)$ and $F \cup (G \setminus F) \in \calF$, a contradiction).
We now claim that $A = X \setminus F$ is an $\calN_\mu$-essential supremum of $\calE$.
Indeed:
\begin{enumerate}
\item Given $E \in \calE$, $\mu(E \setminus A) = \mu(E \cap F) = 0$ since $F \in \calF$;
\item If $B \in \calA$ is such that $\mu(E \setminus B)=0$ for every $E \in \calE$ then $G = X \setminus B \in \calF$ and hence $0 = \mu(G \setminus F) = \mu(A \setminus B)$.
\end{enumerate}
\end{proof}

\begin{Empty}
\label{localized.version}
The first obstacle that comes to mind for a measurable space with negligibles $(X,\calA,\calN)$ to be localizable is that $\calA$ is not required to be stable under arbitrary unions.
If it were, then condition (1) of \ref{def.ess.sup} would be obviously satisfied with $A = \cup \calE \in \calA$.
This is not the end of the story however as condition (2) may well fail for such choice of $A$. 
In fact we give an example below \ref{CR.1} of a measurable space with negligibles of the type $(X,\calP(X),\calN)$ which is consistently not localizable.
Worse yet we exhibit a proper $\sigma$-algebra $\calB \subset \calP(X)$ such that $(X,\calA,\calA \cap \calN)$ is consistently non localizable whenever $\calB \subset \calA \subset \calP(X)$ is a $\sigma$-algebra.
This ruins the hope that with each measurable space with negligibles one can associate a localizable version of it by <<adding enough measurable sets>> to the given $\sigma$-algebra.
%
%Accepting to enlarge the base set $X$, see \ref{category}(4) for resurrecting some weak hope.
\end{Empty}

\begin{Empty}
\label{last.lemma}
{\it 
Assume $(X,\calA,\calN)$ and $(Y,\calB,\calM)$ are measurable spaces with negligibles and $f : X \to Y$ is a bijection such that
\begin{equation*}
\calA = \calP(X) \cap \{ f^{-1}(B) : B \in \calB \}
\end{equation*}
and
\begin{equation*}
\calN = \calP(X) \cap \{ f^{-1}(M) : M \in \calM \} \,.
\end{equation*}
It follows that $(X,\calA,\calN)$ is localizable if and only if $(Y,\calB,\calM)$ is localizable.
}
\par 
This can be checked directly by routine verification from the definition of essential supremum or by observing that the quotient Boolean algebras $\calA_{\calN}$ and $\calB_{\calM}$ are isomorphic and referring to \ref{abstract.loc}. 
%
%Such $f$ is an instance of an isomorphism in the category to be discussed in \ref{category}.
%
We will use this result in \ref{pue.8} below.
\end{Empty}

\begin{Empty}
If $(X,\calA)$ is a measurable space and $E \in \calA$ we associate with it its subspace $(E,\calA_E)$ where $\calA_E = \calP(E) \cap \{E \cap A : A \in \calA\}$.
\end{Empty}

\begin{Empty}
Let $(X,\calA,\calN)$ be a measurable space with negligibles and let $\calE \subset \calA$.
A {\em family subordinated to $\calE$} is a family $(g_E)_{E \in \calE}$ such that
\begin{enumerate}
\item $g_E : E \to \R$ is $\calA_E$-measurable for every $E \in \calE$.
\end{enumerate}
We further say that $(g_E)_{E \in \calE}$ is {\em compatible} if also
\begin{enumerate}
\item[(2)] For every pair $E_1,E_2 \in \calE$ one has $E_1 \cap E_2 \cap \{ g_{E_1} \neq g_{E_2} \} \in \calN$.
\end{enumerate}
A {\em gluing} of a compatible family $(g_E)_{E \in \calE}$ subordinated to $\calE$ is a function $g : X \to \R$ such that
\begin{enumerate}
\item[(3)] $g$ is $\calA$-measurable;
\item[(4)] $E \cap \{ g \neq g_E \} \in \calN$ for every $E \in \calE$.
\end{enumerate}
\end{Empty}

\begin{Proposition}
\label{gluing}
Let $(X,\calA,\calN)$ be a measurable space with negligibles. 
The following are equivalent.
\begin{enumerate}
\item $(X,\calA,\calN)$ is localizable.
\item For every $\calE \subset \calA$, every compatible family subordinated to $\calE$ admits a gluing.
\end{enumerate}
\end{Proposition}

\begin{proof}
$(1) \Rightarrow (2)$ Let $\calE \subset \calA$ and let $(g_E)_{E \in \calE}$ be a compatible family subordinated to $\calE$.
With each $q \in \Q$ and $E \in \calE$ we associate $E_q = E \cap \{ g_E \geq q \} \in \calA$.
Thus given $q \in \Q$ the family $\{ E_q : E \in \calE\}$ admits an $\calN$ essential supremum which we denote as $A_q \in \calA$.
Define $\bar{g} : X \to [-\infty,+\infty]$ by the formula $\bar{g}(x) = \sup \{ q : x \in A_q\}$, $x \in X$, where as usual $\inf \emptyset = - \infty$. 
Notice that if $q \in \Q$ then $\{ \bar{g} > q \} = \cup_{\substack{r \in \Q\\r > q}}A_r \in \calA$, thus $\bar{g}$ is $\calA$-measurable. 
\par 
Given $E \in \calE$ we shall now establish that
\begin{equation}
\label{eq.1}
E \cap \{ g_E \neq \bar{g} \} \in \calN \,.
\end{equation}
If $x \in E \cap \{ g_E < \bar{g}\}$ then there exists $q \in \Q$ such that $g_E(x) < q$ and $x \in A_q$. Accordingly,
\begin{equation}
\label{eq.2}
E \cap \{ g_E < \bar{g} \} \subset \cup_{q \in \Q} E \cap (A_q \setminus E_q ) \,.
\end{equation}
Now if $q \in \Q$ and $E' \in \calE$ then
\begin{equation*}
E'_q \setminus (E^c \cup E_q) = E \cap \{ g_E < q \} \cap E' \cap \{ g_{E'} \geq q \} \subset E \cap E' \cap \{ g_E \neq g_{E'} \} \in \calN \,.
\end{equation*}
Since $E' \in \calE$ is arbitrary we infer that
\begin{equation*}
\calN \ni A_q \setminus (E^c \cup E_q) = E \cap (A_q \setminus E_q)
\end{equation*}
and it therefore ensues from \eqref{eq.2} that
\begin{equation}
\label{eq.3}
E \cap \{ g_E < \bar{g} \} \in \calN \,.
\end{equation}
\par 
Next if $x \in E \cap \{ g_E > \bar{g} \}$ then there exists $q \in \Q$ such that $g_E(x) > q$ and $x \not\in A_q$.
Consequently,
\begin{equation}
\label{eq.4}
E \cap \{ g_E > \bar{g} \} \subset \cup_{q \in \Q} E \cap (E_q \setminus A_q) \in \calN \,.
\end{equation}
It now follows from \eqref{eq.3} and \eqref{eq.4} that \eqref{eq.1} holds.
\par 
Finally we let $A = \{ \bar{g} \in \R \} \in \calA$ and $g = \bar{g}.\ind_A$ (with the usual convention that $(\pm \infty).0=0$). Thus $g$ is $\calA$-measurable and, for each $E \in \calE$, $E \cap \{ g \neq g_E \} \subset E \cap \{ \bar{g} \neq g_E \} \in \calN$ since $g_E$ is $\R$ valued.
Whence $g$ is a gluing of $(g_E)_{E \in \calE}$.
\par 
$(2) \Rightarrow (1)$ Let $\calE \subset \calA$ and define $\calE^* = \calE \setminus \calN$ as well as
\begin{equation*}
\calF = \calA \cap \{ F : F \cap E \in \calN \text{ for every } E \in \calE^* \} \,.
\end{equation*}
Notice that $\calF \cap \calE^* = \emptyset$. 
Put $\calG = \calE^* \cup \calF$ and define a family $(g_G)_{G \in \calG}$ subordinated to $\calG$ as follows.
If $E \in \calE^*$ then $g_E = \ind_E$, and if $F \in \calF$ then $g_F = 0.\ind_F$. 
One easily checks that $(g_G)_{G \in \calG}$ is a compatible family, thus it admits a gluing $g$ by assumption.
Let $A = \{ g = 1 \} \in \calA$.
We ought to show that $A$ is an $\calN$ essential supremum of $\calE$. 
\par 
First let $E \in \calE$. 
If $E \in \calN$ then clearly $E \setminus A \in \calN$.
Otherwise $E \in \calE^*$ and hence $E \setminus A = E \cap \{ g \neq 1 \} = E \cap \{ g \neq g_E \} \in \calN$.
Suppose now that $B \in \calE$ is such that $E \setminus B \in \calN$ for every $E \in \calE$.
Let $F = A \setminus B \in \calA$.
Given $E \in \calE$ we observe that $F \cap E = E \cap (A \setminus B) \subset A \setminus B \in \calN$.
Therefore $F \in \calF$.
It follows that $A \setminus B = F \cap A = F \cap \{ g = 1 \} \subset F \cap \{ g \neq g_F \} \in \calN$ and the proof is complete. 
\end{proof}

\begin{Empty}
\label{localized.negligibles}
Given a measurable space with negligibles $(X,\calA,\calN)$ and $\calF \subset \calA$ an arbitrary family, we define
\begin{equation*}
\calN[\calF] = \calA \cap \{ A : A \cap F \in \calN \text{ for every } F \in \calF \} \,.
\end{equation*}
The following are immediate consequences of the definition.
{\it 
\begin{enumerate}
\item $\calN[\calF]$ is a $\sigma$-ideal in $\calA$.
\item $\calN \subset \calN[\calF]$.
\item If $\calF_1 \subset \calF_2$ then $\calN[\calF_1] \supset \calN[\calF_2]$.
\item If $g$ and $g'$ are both gluings of a compatible family $(g_E)_{E \in \calE}$ subordinated to $\calE$ then $\{ g \neq g' \} \in \calN[\calE]$.
\end{enumerate}
}
\end{Empty}

\begin{Empty}
\label{ideals}
Let $(X,\calA,\calN)$ be a measurable space with negligibles. 
We say that $\calI \subset \calA$ is an {\em ideal} in $\calA$ whenever the following holds:
\begin{enumerate}
\item $\emptyset \in \calI$;
\item If $A \in \calA$, $B \in \calI$ and $A \subset B$ then $A \in \calI$;
\item If $A_1,\ldots,A_N \in \calI$ then $\cup_{n=1}^N A_n \in \calI$.
\end{enumerate}
One observes that each $\calE \subset \calA$ is contained in a smallest ideal which we denote as $\rmideal(\calE)$.
The reader will easily check that
\begin{equation*}
\rmideal(\calE) = \calA \cap \{ A : \text{ there exist } E_1,\ldots,E_n \in \calE \text{ such that } A \subset \cup_{n=1}^N E_n \} \,.
\end{equation*}
Therefore
{\it 
\begin{enumerate}
\item[(4)] $\calN[\calE] = \calN[\rmideal(\calE)]$.
\item[(5)] $A \in \calA$ is an $\calN$ essential supremum of $\calE$ if and only if $A$ is an $\calN$ essential supremum of $\rmideal(\calE)$.
\end{enumerate}
}
There is no difficulty in showing that the latter is a consequence of the definition of essential supremum and of the following claim: If $C \in \calA$ is such that $E \setminus C \in \calN$ for every $E \in \calE$ then also $F \setminus C \in \calN$ for every $F \in \rmideal(\calE)$.
\end{Empty}

\begin{Empty}[Partition of unity]
Let $(X,\calA,\calN)$ be a measurable space with negligibles, and $\calI \subset \calA$ an ideal.
A {\em partition of unity relative to $\calI$} is a collection $\calE \subset \calI$ such that
\begin{enumerate}
\item $\calE \cap \calN = \emptyset$;
\item For every $E_1,E_2 \in \calE$, if $E_1 \neq E_2$ then $E_1 \cap E_2 \in \calN$;
\item For every $A \in \calI \setminus \calN$ there exists $E \in \calE$ such that $A \cap E \not\in \calN$.
\end{enumerate}
\end{Empty}

\begin{Lemma}
\label{existence.pu}
Let $(X,\calA,\calN)$ be a measurable space with negligibles, and $\calI \subset \calA$ an ideal.
There exists a partition of unity $\calE$ relative to $\calI$.
Furthermore $\calE \neq \emptyset$ in case $\calI \setminus \calN \neq \emptyset$.
\end{Lemma}

\begin{proof}
This is a routine application of Zorn's Lemma.
\end{proof}

\begin{Empty}[Magnitude]
\label{magnitude}
Let $(X,\calA,\calN)$ be a measurable space with negligibles, $\calI$ an ideal in $\calA$, and $\kappa$ a cardinal.
We say that $(X,\calA,\calN)$ {\em has magnitude less than $\kappa$ relative to $\calI$} whenever 
for every $\calE \subset \calI$ with the following properties:
\begin{enumerate}
\item $\calE \cap \calN = \emptyset$;
\item For every $E_1,E_2 \in \calE$, if $E_1 \neq E_2$ then $E_1 \cap E_2 \in \calN$;
\end{enumerate}
one has $\rmcard \calE \leq \kappa$.
\end{Empty}

\section{Semifinite and Semilocalizable Measure Spaces}
\label{MS}

For a measure space $(X,\calA,\mu)$ the canonical map from $\bL_\infty(X,\calA,\mu)$ to $\bL_1(X,\calA,\mu)^*$ is in general neither injective nor surjective.
It is known that injectivity is equivalent to semifiniteness of $(X,\calA,\mu)$.
We identify here a condition equivalent to surjectivity, which we call semilocalizability.

\begin{Empty}
\label{upsilon}
Let $(X,\calA,\mu)$ be a measure space.
We consider the map
\begin{equation*}
\Upsilon : \bL_\infty(X,\calA,\mu) \to \bL_1(X,\calA,\mu)^*
\end{equation*}
defined in the following way.
Given $\bg \in \bL_\infty(X,\calA,\mu)$ and $\boldf \in \bL_1(X,\calA,\mu)$ we let $\Upsilon(\bg)(\boldf) = \int_X gf d\mu$ for a choice of $g \in \bg$ and $f \in \boldf$. 
In general $\Upsilon$ is neither injective nor surjective.
In this Section we state a necessary and sufficient condition for $\Upsilon$ to be injective (namely that the measure space be semifinite), and a necessary and sufficient condition for $\Upsilon$ to be surjective (namely that the measure space be semilocalizable).  
\end{Empty}

\begin{Empty}
We say that a measure space $(X,\calA,\mu)$ is {\em semifinite} whenever the following holds: For every $A \in \calA$ such that $\mu(A) = \infty$ there exists $B \in \calA$ such that $B \subset A$ and $0 < \mu(B) < \infty$.
Clearly all $\sigma$-finite measure spaces are semifinite.
We recall that $\Upsilon$ is injective if and only if $(X,\calA,\mu)$ is semifinite and in that case $\Upsilon$ is an isometry, \cite[243G(a)]{FREMLIN.II}.
Furthermore if $(X,\calA,\mu)$ is semifinite then $\Upsilon$ is bijective if and only if $(X,\calA,\calN_\mu)$ is localizable, \cite[243G(b)]{FREMLIN.II}\footnote{\textsc{D.H. Fremlin} calls localizable a measure space $(X,\calA,\mu)$ which is semifinite and such that $(X,\calA,\calN_\mu)$ is (in the vocabulary introduced in the present paper) a localizable measurable space with negligibles.}.
Below we give a necessary and sufficient condition for $\Upsilon$ to be surjective (not assuming that it be injective in the first place).
This seems to be new.
\end{Empty}

\begin{Empty}
Let $(X,\calA,\mu)$ be a measure space.
We say that $A \in \calA$ is {\em purely infinite} if for every $B \in \calA$ such that $B \subset A$  one has $\mu(B) = 0$ or $\mu(B) = \infty$.
Thus $(X,\calA,\mu)$ is semifinite if and only if there exists no purely infinite $A \in \calA$.
We define
\begin{equation*}
\begin{split}
\calN_\mu & = \calA \cap \{ A : \mu(A)=0 \} \\
\calA^f_\mu & = \calA \cap \{ A : \mu(A) < \infty \} \\
\calA^{pi}_\mu & = \calA \cap \{ A : A \text{ is purely infinite } \} \,,
\end{split}
\end{equation*}
and we abbreviate $\calA^f = \calA^f_\mu$ and $\calA^{pi} = \calA^{pi}_\mu$ when no confusion occurs, which is almost always.
Clearly $\calA^f$ is an ideal, whereas $\calN_\mu$ and $\calN_\mu \cup \calA^{pi}$ are $\sigma$-ideals.
In fact one easily checks that $\calN_\mu[\calA^f] = \calN_\mu \cup \calA^{pi}$ and also that $\calN_\mu \cup \calA^{pi} = \calN_{\mu_{sf}}$ where $(X,\calA,\mu_{sf})$ is the semifinite version of $(X,\calA,\mu)$, see \cite[213X(c)]{FREMLIN.II} and also \ref{sigmafin} below.
\par
Referring to \ref{localized.negligibles} we consider the $\sigma$-ideal $\calN_\mu[\calA^f]$ which will play the major role in the present Section.
When we need to refer to its members we call these {\em locally $\mu$ null}.
\end{Empty}

\begin{Lemma}
\label{sigmafin}

Let $(X,\calA,\mu)$ be a measure space. The following are equivalent:
\begin{enumerate}
\item $(X,\calA,\mu)$ is semifinite;
\item $\calN_\mu = \calN_\mu[\calA^f]$.
\end{enumerate}
\end{Lemma}

\begin{proof}
It follows from \ref{localized.negligibles}(1) that (2) is equivalent to $\calN_\mu[\calA^f] \subset \calN_\mu$.
Therefore $\neg (2)$ is equivalent to $\calN_\mu[\calA^f] \setminus \calN_\mu \neq \emptyset$.
One easily observes that $\calN_\mu[\calA^f] \setminus \calN_\mu = \calA^{pi}$.
Since $(X,\calA,\mu)$ is semifinite if and only if $\calA^{pi}=\emptyset$ the proof is complete.
\end{proof}

\begin{Empty}
\label{def.semilocalizable}
A measure space $(X,\calA,\mu)$ is called {\em semilocalizable} if the measurable space with negligibles $(X,\calA,\calN_\mu[\calA^f])$ is localizable.
\end{Empty}

\begin{Proposition}
\label{Riesz}
Let $(X,\calA,\mu)$ be a measure space.
The following are equivalent.
\begin{enumerate}
\item $(X,\calA,\mu)$ is semilocalizable.
\item $\Upsilon$ is surjective.
\end{enumerate}
\end{Proposition}

\begin{proof}
$(1) \Rightarrow (2)$ To each $E \in \calA^f$ we associate the linear isometry $\bbeta_E : \bL_1(E,\calA_E,\mu_E) \to \bL_1(X,\calA,\mu)$ defined in the obvious way (extending $\boldf \in \bL_1(E,\calA_E,\mu_E)$ by zero outside of $E$), as well as the linear map $\brho_E : \bL_1(X,\calA,\mu) \to \bL_1(E,\calE,\mu_E)$ (restricting $\boldf \in \bL_1(X,\calA,\mu)$ to $E$).
Thus $(\brho_E \circ \bbeta_E)(\boldf) = \boldf$ for every $\boldf \in \bL_1(E,\calA_E,\mu_E)$ and $(\bbeta_E \circ \brho_E)(\boldf) = \boldf$ for every $\boldf \in \bL_1(X,\calA,\mu)$ such that $E^c \cap \{ f \neq 0 \} \in \calN_\mu$, $f \in \boldf$.
Given $\balpha \in \bL_1(X,\calA,\mu)^*$ and $E \in \calA^f$ it follows that $\balpha \circ \bbeta_E \in \bL_1(X,\calA_E,\mu_E)^*$.
Since $(E,\calA_E,\mu_E)$ is a finite measure space the classical Riesz Representation Theorem yields an $\calA_E$-measurable function $g_E : E \to \R$ such that $(\balpha \circ \bbeta_E)(\boldf) = \int_X g_E f d\mu_E$ for every $f \in \boldf \in \bL_1(E,\calA_E,\mu_E)$ and $\sup |g_E| \leq \| \balpha \circ \bbeta_E \| \leq \| \balpha \|$.
We shall now observe that the family $(g_E)_{E \in \calA^f}$ is compatible.
Let $E_1,E_2 \in \calA^f$, $n \in \N^*$ and define $Z_n = E_1 \cap E_2 \cap \{ g_{E_1} \leq - n^{-1} + g_{E_2} \}$. Thus $f_n=\ind_{Z_n} \in L_1(X,\calA,\mu)$ and
\begin{multline*}
\balpha(\boldf_n) = (\balpha \circ \bbeta_{E_1} \circ \brho_{E_1})(\boldf_n) = \int_{E_1} g_{E_1} f_n d\mu_{E_1} = \int_{Z_n} g_{E_1} d\mu \\ \leq - n^{-1} \mu(Z_n) + \int_{Z_n} g_{E_2} d\mu = - n^{-1} \mu(Z_n) + \int_{E_2} g_{E_2} f_n d\mu_{E_2} \\ = - n^{-1} \mu(Z_n) + (\balpha \circ \bbeta_{E_2} \circ \brho_{E_2})(\boldf_n) = - n^{-1} \mu(Z_n) + \balpha(\boldf_n) \,.
\end{multline*}
Therefore $\mu(Z_n)=0$, thus also $E_1 \cap E_2 \cap \{ g_{E_1} < g_{E_2} \} = \cup_{n \in \N^*} Z_n \in \calN_\mu$, and in turn $E_1 \cap E_2 \cap \{ g_{E_1} \neq g_{E_2} \} \in \calN_\mu \subset \calN_\mu[\calA^f]$.
Since $(X,\calA,\calN_\mu[\calA^f])$ is localizable by assumption it follows from \ref{gluing} that $(g_E)_{E \in \calA^f}$ admits a gluing $\tilde{g} : X \to \R$. 
We let $Z = X \cap \{ |\tilde{g}| > \|\balpha\| \} \in \calA$.
It ensues from our choice of a special representative $g_E$ that $E \cap Z \subset E \cap \{ \tilde{g} \neq g_E \} \in \calN_\mu[\calA^f]$ for every $E \in \calA^f$.
Hence the function $g = \tilde{g}.\ind_{Z^c}$ is also a gluing of $(g_E)_{E \in \calA^f}$, and furthermore $\sup |g| \leq \|\balpha\| < \infty$.
Therefore $\bg \in \bL_\infty(X,\calA,\mu)$ and it remains to establish that $\Upsilon(\bg) = \balpha$.
\par 
Let $f \in \boldf \in \bL_1(X,\calA,\mu)$, define $A = \{ f \neq 0 \}$, and $A_n = \{ |f| \geq n^{-1} \}$, $n \in \N^*$. 
Thus $A = \cup_{n \in \N^*} A_n$ and $A_n \in \calA^f$ for each $n \in \N^*$.
Letting $f_n = f.\ind_{A_n}$ we notice that $(\boldf_n)_{n \in \N^*}$ converges to $\boldf$ in $\bL_1(X,\calA,\mu)$, whence $\lim_n \balpha(\boldf_n) = \balpha(\boldf)$.
We also notice that $gf_n \to gf$ as $n \to \infty$, everywhere, and that $|gf_n| \leq |gf| \in L_1(X,\calA,\mu)$, so that the Dominated Convergence Theorem applies to $(gf_n)_{n \in \N^*}$.
Accordingly,
\begin{equation*}
\begin{split}
\lim_n \balpha(\boldf_n) & = \lim_n (\balpha \circ \bbeta_{A_n} \circ \brho_{A_n} )(\boldf_n) \\
& = \lim_n \int_{A_n} g_{A_n} f_n d\mu_{A_n} \\
& = \lim_n \int_{A_n} g f_n d\mu \quad 
\text{(because $A_n \cap \{ g \neq g_{A_n} \} \in \calN_\mu$ according to \ref{sigmafin})}\\
& = \lim_n \int_X g f_n d\mu \\
& = \int_X gf d\mu \\
& = \Upsilon(\bg)(\boldf) \,. 
\end{split}
\end{equation*}
\par 
$(2) \Rightarrow (1)$ Let $\calE \subset \calA$. 
We ought to show that $\calE$ admits an $\calN_\mu[\calA^f]$-essential supremum in $\calA$.
According to \ref{ideals}(5) there is no restriction to assume that $\calE$ is an ideal.
We will define some $\balpha \in \bL_1(X,\calA,\mu)^*$ associated with $\calE$.
We start by defining $\alpha(f) \in \R_+$ associated with $f \in L_1(X,\calA,\mu)$, $f \geq 0$, by the following formula:
\begin{equation*}
\alpha(f) = \sup_{E \in \calE} \int_E f d\mu \,.
\end{equation*} 
We claim that the following hold:
\begin{enumerate}
\item[(a)] For every $f \in L_1(X,\calA,\mu)^+$ one has $0 \leq \alpha(f) \leq \int_X f d\mu < \infty$;
\item[(b)] For every $f_1,f_2 \in L_1(X,\calA,\mu)^+$ one has $\alpha(f_1+f_2) = \alpha(f_1) + \alpha(f_2)$;
\item[(c)] For every $f \in L_1(X,\calA,\mu)^+$ and every $t \geq 0$ one has $\alpha(t.f) = t.\alpha(f)$;
\item[(d)] For every $f_1,f_2,f'_1,f_2' \in L_1(X,\calA,\mu)^+$ if $f_1-f_2 = f'_1-f'_2 \geq 0$ then 
$\alpha(f_1) - \alpha(f_2) = \alpha(f'_1) - \alpha(f'_2)$.
\end{enumerate}
Claims (a) and (c) are obvious.
For proving (d) we notice that $\alpha(f_1) + \alpha(f'_2) = \alpha(f_1+f'_2) = \alpha(f'_1+f_2) = \alpha(f'_1) + \alpha(f_2)$, according to (b).
Regarding (b) we first notice that $\alpha(f_1+f_2) \leq \alpha(f_1) + \alpha(f_2)$. 
Furthermore given $\veps > 0$ there are $E_j \in \calE$, $j=1,2$, such that $\alpha(f_j) \leq \veps + \int_{E_j} f_j d\mu$. 
Therefore
\begin{multline*}
\alpha(f_1) + \alpha(f_2) \leq 2\veps + \int_{E_1} f_1 d \mu + \int_{E_2} f_2 d \mu \leq 2\veps + \int_{E_1 \cup E_2} f_1 d \mu + \int_{E_1 \cup E_2} f_2 d \mu \\ = 2\veps + \int_{E_1 \cup E_2} (f_1 + f_2) d\mu \leq 2\veps + \alpha(f_1+f_2)
\end{multline*}
because $E_1 \cup E_2 \in \calE$.
Since $\veps > 0$ is arbitrary, claim (b) follows.
\par 
Now if $f \in L_1(X,\calA,\mu)$ we define $\alpha(f) = \alpha(f^+) - \alpha(f^-) \in \R$ -- a definition compatible with the previous one when $f \geq 0$.
It easily follows from (d) that $\alpha$ is additive.
Observing that $\alpha(-f) = - \alpha(f)$ when $f \geq 0$, it follows from (c) that $\alpha$ is homogeneous of degree 1.
In other words $\alpha$ is linear.
Furthermore (a) implies that $|\alpha(f)| \leq \int_X |f|d\mu$.
It is now clear that $\balpha(\boldf) = \alpha(f)$, $f \in \boldf \in \bL_1(X,\calA,\mu)$, is well defined and that $\balpha \in \bL_1(X,\calA,\mu)^*$.
\par 
It ensues from the hypothesis that there exists $g \in \bg \in \bL_\infty(X,\calA,\mu)$ such that
\begin{equation*}
\int_X gf d\mu = \balpha(\boldf) = \alpha(f) = \sup_{E \in \calE} f d\mu
\end{equation*}
for all $0 \leq f \in \boldf \in \bL_1(X,\calA,\mu)^+$.
We define $A = \{ g \neq 0 \} \in \calA$ and we will next check that $A$ is an $\calN_\mu[\calA^f]$ essential supremum of $\calE$.
\par 
Let $E \in \calE$. 
Define $Z = E \setminus A = E \cap \{ g = 0 \}$.
Given $F \in \calA^f$ let $\ind_{F \cap Z} \in L_1(X,\calA,\mu)^+$.
Thus
\begin{equation*}
0 = \int_X g\ind_{F \cap Z} d\mu = \alpha(\ind_{F \cap Z}) \geq \int_E \ind_{F \cap Z} d\mu = \mu(F \cap (E \setminus A)) \,.
\end{equation*}
Since $F \in \calA^f$ is arbitrary it follows that $E \setminus A \in \calN_\mu[\calA^f]$.
\par 
We next claim that if $F \in \calA^f$ then $F \cap \{ g < 0 \} \in \calN_\mu$.
Letting $Z_n = \{ g \leq - n^{-1} \}$, $n \in \N^*$, we notice that $\ind_{F \cap Z_n} \in L_1(X,\calA,\mu)^+$ whence
\begin{equation*}
- n^{-1} \mu(F \cap Z_n) \geq \int_X g \ind_{F \cap Z_n} d\mu = \alpha(\ind_{F \cap Z_n}) \geq 0 \,.
\end{equation*}
Thus clearly $\mu(F \cap Z_n)=0$ and, since $n \in \N^*$ is arbitrary $\mu(F \cap \{ g < 0 \})=0$.
\par 
Finally we assume that $B \in \calA$ is such that $\mu(F \cap (E \setminus B))=0$ for every $F \in \calA^f$.
We define $Z = A \setminus B \in \calA$. 
Let $F \in \calA^f$ and notice once again that $\ind_{F \cap Z} \in L_1(X,\calA,\mu)^+$, therefore
\begin{equation*}
\int_X g \ind_{F \cap Z} d \mu = \alpha(\ind_{F \cap Z}) = \sup_{E \in \calE} \int_E \ind_{F \cap Z} d\mu.
\end{equation*}
Now given $E \in \calE$ we observe that $E \cap F \cap Z = F \cap E \cap (A \setminus B) \subset F \cap (E \setminus B) \in \calN_\mu$ by our assumption about $B$. 
Since $E \in \calE$ is arbitrary we infer that
\begin{equation*}
\int_X g \ind_{F \cap Z} d\mu = 0 \,.
\end{equation*}
It follows from the previous paragraph and the definition of $Z$ that $g > 0$, $\mu$ almost everywhere on $Z \cap F$.
Consequently $\mu(F \cap Z)=0$ and the proof is complete.
\end{proof}

\begin{Empty}
\label{dixmier}
Let $(X,\calA,\calN)$ be a measurable space with negligibles.
Here we recall that if $\bL_\infty(X,\calA,\calN)$ is isometrically isomorphic to a dual Banach space then $(X,\calA,\calN)$ is localizable -- indeed in this case $C(\sfSpec(\calA_{\calN}))$ is isometrically isomorphic to a dual Banach space according to \ref{LC}, whence it is isometrically injective \cite[4.3.8(i)]{ALBIAC.KALTON}, and it remains to recall \ref{abstract.loc}.
However if $(X,\calA,\calN)$ is localizable then $\bL_\infty(X,\calA,\calN)$ does not need to be isometrically isomorphic to a dual Banach space in general (see \cite[Problems 4.8 and 4.9]{ALBIAC.KALTON} for an example due to R. Dixmier), yet below we show the conditions are equivalent in the class of measurable spaces with negligibles of the type $(X,\calA,\calN_\mu[\calA^f])$ for some measure space $(X,\calA,\mu)$.
\end{Empty}

\begin{Proposition}
Let $(X,\calA,\mu)$ be a measure space.
The following are equivalent.
\begin{enumerate}
\item $(X,\calA,\calN_\mu[\calA^f])$ is localizable;
\item $\bL_\infty(X,\calA,\calN_\mu[\calA^f])$ is isometrically isomorphic to a dual Banach space.
\end{enumerate}
In this case $\bL_\infty(X,\calA,\calN_\mu[\calA^f])$ is isometrically isomorphic to $\bL_1(X,\calA,\mu)^*$.
\end{Proposition}

\begin{proof}
That $(2) \Rightarrow (1)$ follows from the general argument in \ref{dixmier}.
We henceforth assume that $(X,\calA,\calN_\mu[\calA^f])$ is localizable and we let $L_\infty(X,\calA)$ denote the linear space consisting of those {\it bounded}, $\calA$ measurable functions $g : X \to \R$.
In the exact same way as in \ref{upsilon} we define a linear map
\begin{equation*}
\hat{\Upsilon} : L_\infty(X,\calA) \to \bL_1(X,\calA,\mu)^* \,.
\end{equation*}
\par 
We claim that $\ker \hat{\Upsilon}$ consists of those $g \in L_\infty(X,\calA)$ such that $S_g = \{ g \neq 0 \} \in \calN_\mu[\calA^f]$.
If $g$ has this property and $f \in L_1(X,\calA,\mu)$ then $\{ f \neq 0 \} = \cup_{n \in \N^*} \{ n \leq |f| \}$ and since each $\{ n \leq |f| \} \in \calA^f$ it follows that $\{ gf \neq 0 \} = \{ g \neq 0 \} \cap \{ f \neq 0 \} \in \calN_\mu$ and in turn $\hat{\Upsilon}(g)(f) = \int_X gf d\mu =0$.
The other way around we let $g \in \ker \hat{\Upsilon}$ and we define $S^\pm = \{ \pm g > 0 \}$ so that $\{ g \neq 0 \} = S^+ \cup S^-$.
Given $A \in \calA^f$ and letting $f = \ind_{A \cap S^+}$ we infer that $0 = \hat{\Upsilon}(g)(f) = \int_A g^+ d\mu$ thus $S^+ \cap A \in \calN_\mu$. 
Thus $S^+ \in \calN_\mu[\calA^f]$, and similarly $S^- \in \calN_\mu[\calA^f]$.
\par
Since $ \bL_1(X,\calA,\mu)^*$ is linearly isomorphic to $L_\infty(X,\calA)/\ker \hat{\Upsilon}$,
the claim being established we now easily infer that $ \bL_1(X,\calA,\mu)^*$ is linearly isomorphic to $\bL_\infty(X,\calA,\calN_\mu[\calA^f])$.
It remains to show that the corresponding linear isomorphism associated with $\hat{\Upsilon}$ is an isometry.
We leave the details to the reader.
\end{proof}

%\section{Radon-Nikod\'ym Derivatives}

\section{Almost Decomposable Measure Spaces}

In this Section we state basic facts on the notion of almost decomposable measure space introduced in \cite{DEP.98}.
It is an appropriate generalization to non semifinite measure spaces of the notion of decomposable measure space (also called strictly localizable measure space).
I learned \ref{CH.implies.ad} from \cite[2.5.10]{GMT} (in a different language than here).
I learned the idea in \ref{55} from \textsc{D.H. Fremlin}.

\begin{Empty}
\label{almost.decomposition}
Let $(X,\calA,\mu)$ be a measure space. An {\em almost decomposition} of $(X,\calA,\mu)$ is a family $\calG \subset \calA$ with the following properties:
\begin{enumerate}
\item $\forall G \in \calG : \mu(G) < \infty$;
\item $\calG$ is disjointed;
\item $\forall A \in \calP(X) : ( \forall G \in \calG : A \cap G \in \calA) \Rightarrow A \in \calA$;
\item $\forall A \in \calA : \mu(A) < \infty \Rightarrow \mu(A) = \sum_{G \in \calG} \mu(A \cap G)$.
\end{enumerate}
We say that $(X,\calA,\mu)$ is {\em almost decomposable} if it admits an almost decomposition.
\end{Empty}

\begin{Empty}
\label{52}
Almost decomposable measure spaces generalize $\sigma$-finite measure spaces.
In fact, assuming that $(X,\calA,\mu)$ is semifinite, if $\calG$ is an almost decomposition of $(X,\calA,\mu)$ and $\calG$ is (at most) countable then $\mu$ is $\sigma$-finite.
Indeed $S = \cup \calG \in \calA$ (either because $\calG$ is countable or according to \ref{almost.decomposition}(3)), and we ought to show that $\mu(X \setminus S) = 0$.
If $A \in \calA$, $\mu(A) < \infty$ and $A \subset X \setminus S$ then $\mu(A)=0$ according to \ref{almost.decomposition}(4).
Since $\mu$ is semifinite this implies $\mu(Z \setminus S)=0$.
\end{Empty}

\begin{Proposition}
\label{ad.implies.semiloc}
If a measure space admits an almost decomposition then it is semilocalizable.
\end{Proposition}

\begin{proof}
Let $\calG$ be an almost decomposition of the measure space $(X,\calA,\mu)$ and let $\calE \subset \calA$ be an arbitrary family.
With each $G \in \calG$ we associate $\calE_G = \{ G \cap E : E \in \calE \} \subset \calA_G$.
Since $(G,\calA_G,\mu_G)$ is a finite measure space, $(G,\calA_G,\calN_{\mu_G})$ is localizable according to \ref{sigmafinite.localizable} and we let $A_G \in \calA_G \subset \calA$ be an $\calN_{\mu_G}$ essential supremum of $\calE_G$.
We now define a subset of $X$
\begin{equation*}
A = \cup_{G \in \calG} A_G \,.
\end{equation*}
Since $A \cap G = A_G \in \calG$ for every $G \in \calG$ it follows from condition (3) of the definition of an almost decomposition that $A \in \calA$.
\par 
We shall now show that $E \setminus A \in \calN_\mu[\calA^f]$ for every $E \in \calE$.
Let $E \in \calE$ and $F \in \calA^f$.
It follows that
\begin{equation*}
\begin{split}
\mu [ F \cap (E \setminus A) ] & = \sum_{G \in \calG} \mu [ F \cap G \cap (E \setminus A) ] \quad\quad \text{(by \ref{almost.decomposition}(4))}\\
& \leq \sum_{G \in \calG} \mu [ (E \cap G) \setminus A_G  ] \\
& = 0 \,.
\end{split}
\end{equation*}
\par 
Finally we assume that $B \in \calA$ is such that $E \setminus B \in \calN_\mu[\calA^f]$ for all $E \in \calE$ and we ought to show that $A \setminus B \in \calN_\mu[\calA^f]$.
Let $F \in \calA^f$.
We must show that $\mu [ F \cap (A \setminus B) ] = 0$.
Notice that
\begin{equation*}
\begin{split}
\mu[F \cap (A \setminus B) ] & = \sum_{G \in \calG} \mu [ F \cap G \cap (A \setminus B) ] \quad\quad \text{(by \ref{almost.decomposition}(4))} \\
& \leq \sum_{G \in \calG} \mu [G \cap (A \setminus B) ] \,.
\end{split}
\end{equation*}
Thus it suffices to establish that $\mu[G \cap (A \setminus B)]=0$ for each $G \in \calG$.
Fix $G \in \calG$ and let $B_G = B \cap G$.
Thus $\mu [(E \cap G) \setminus B_G] = \mu[ G \cap (E \setminus B) ]=0$ for every $E \in \calE$.
Therefore $\mu(A_G \setminus B_G) = 0$.
Since $A_G \setminus B_G = (A \cap G) \setminus (B \cap G) = G \cap (A \setminus B)$ the proof is complete.
\end{proof}

\begin{Proposition}[$\mathsf{ZFC + CH}$]
\label{CH.implies.ad}
Let $X$ be a Polish space and let $\phi$ be a Borel regular outer measure on $X$ (recall \ref{24}).
Assuming the Continuum Hypothesis, the measure space $(X,\calA_{\phi},\phi)$ admits an almost decomposition.
\end{Proposition}

\begin{proof}
In case $X$ is finite the conclusion clearly holds.
We henceforth assume $X$ is infinite.
In that case $\rmcard \calB(X) = 2^{\aleph_0}$ (the upper bound follows from the fact that Borel sets are Suslin, and Suslin sets are continuous images of closed subsets of a particular Polish space, the Baire space, see e.g. \cite[3.3.18]{SRIVASTAVA}).
We abbreviate $\calB = \calB(X)$ and as usual $\calB^f = \calB \cap \{ B : \phi(B) < \infty \}$.
It now follows from the Continuum Hypothesis that $\calB^f$ admits a well-ordering $\preccurlyeq$ such that every proper initial segment $\calB_B^f = \calB^f \cap \{ C : C \preccurlyeq B \text{ and } C \neq B \}$, $B \in \calB^f$, is at most countable and therefore $\cup \calB_B^f \in \calB$.
With each $B \in \calB^f$ we associate the Borel set $G_B = B \setminus \cup \calB_B^f$.
We claim that $\calG = \{ G_B : B \in \calB^f\}$ is an almost decomposition of $(X,\calA_{\phi},\phi)$.
Conditions (1) and (2) of \ref{almost.decomposition} are readily satisfied.
\par 
In order to check that \ref{almost.decomposition}(3) holds, we let $A \subset X$ be such that $A \cap G$ is $\phi$-measurable for each $G \in \calG$ and we ought to show that $A$ is $\phi$-measurable. 
Let $S \subset X$ be arbitrary. 
We must establish that
\begin{equation*}
\phi(S) \geq \phi(S \cap A) + \phi(S \setminus A) \,.
\end{equation*}
Clearly we may assume that $\phi(S) < \infty$.
We choose $\calB^f \ni B \supset S$ with $\phi(S)=\phi(B)$.
Notice that $B \subset \cup \{ G_C : C \in \calB^f \text{ and } C \preccurlyeq B \}$.
Indeed with each $x \in B$ we associate $C(x) = \min \calB^f \cap \{ C : x \in C \}$ so that $C(x) \preccurlyeq B$ and $x \in G_{C(x)}$.
We now number $G_0,G_1,G_2,\ldots$ the sets $G_C$ corresponding to $C \in \calB^f$ with $C \preccurlyeq B$.
Thus $(G_n)_{n \in \N}$ is a disjointed sequence of Borel sets whose union contains $B$, whence $B = \cup_{n \in \N} B \cap G_n$.
In turn $B \cap A = \cup_{n \in \N} B \cap (G_n \cap A)$ is $\phi$-measurable according to our hypothesis about $A$.
Therefore $B \setminus A = B \cap (B^c \cup A^c) = B \cap (B \cap A)^c$ is also $\phi$-measurable, whence
\begin{equation*}
\phi(S) = \phi(B) = \phi(B \cap A) + \phi(B \setminus A) \geq \phi(S \cap A) + \phi(S \setminus A) \,
\end{equation*}
and the proof of (3) is complete.
\par 
We turn to proving that condition \ref{almost.decomposition}(4) holds.
Let $A \subset X$ be $\phi$-measurable and such that $\phi(A) < \infty$.
Owing to the Borel regularity of $\phi$ there exists $\calB^f \ni B \supset A$ such that $\phi(A)=\phi(B)$.
Associate $(G_n)_{n \in \N}$ with $B$ as above.
It follows that $A = \cup_{n \in \N} A \cap G_n$ and of course each $A \cap G_n$ is $\phi$-measurable.
Therefore
\begin{equation*}
\phi(A) = \sum_{n \in \N} \phi(A \cap G_n) \,.
\end{equation*}
Furthermore if $C \in \calB^f$, $C \neq B$ and $B \preccurlyeq C$ then $A \cap G_C \subset B \cap G_C = \emptyset$ and therefore $\phi(A \cap G_C)=0$.
Consequently
\begin{equation*}
\phi(A) = \sum_{G \in \calG} \phi(A \cap G) \,.
\end{equation*}
\end{proof}

\begin{Proposition}
\label{55}
Assume $X$ is a Polish space and $\mu$ a Borel measure in $X$.
If the measure space $(X,\calB(X),\mu)$ is semifinite and almost decomposable then it is $\sigma$-finite.
\end{Proposition}

\begin{proof}
Assume if possible that $(X,\calB(X),\mu)$ is semifinite and almost decomposable but not $\sigma$-finite.
Letting $\calG$ be an almost decomposition of $\calG$ it would ensue from \ref{52} that $\calG$ is uncountable.
Let $\bkappa = \rmcard \calG$.
It follows from the axiom of choice that there exists $A \subset X$ such that $A \cap G$ is a singleton for each $G \in \calG$.
Thus $\rmcard A = \bkappa$, and $A \in \calB(X)$ according to \ref{almost.decomposition}(3). 
Now, $A$ being an uncountable Suslin subset of a Polish space, $\bkappa = \rmcard A = \cac$, see e.g. \cite[4.3.5]{SRIVASTAVA}.
Furthermore if $B \in \calP(A)$ then for each $G \in \calG$ the set $B \cap G$ is either empty or a singleton, therefore $B \in \calB(X)$ as follows from \ref{almost.decomposition}(3). 
Consequently $2^{\cac} = 2^{\bkappa} = \rmcard \calP(A) \leq \rmcard \calB(X)  = \cac$ (where the last equality was already recalled at the beginning of the proof of \ref{CH.implies.ad}), in contradiction with G. Cantor's Theorem that $\cac < 2^{\cac}$.
\end{proof}

\begin{Corollary}
\label{56}
Let $X$ be an uncountable separable complete metric space, and $0 < d < \infty$. 
It follows that either the measure space $(X,\calB(X),\calH^d)$ is $\sigma$-finite or it is not almost decomposable.
\end{Corollary}

\begin{proof}
Indeed \ref{55} applies because $(X,\calB(X),\calH^d)$ is semifinite according to \textsc{J. Howroyd}'s Theorem \cite{HOW.95} or \cite[471S]{FREMLIN.IV}.
\end{proof}

\section{Locally Determined Measure Spaces of Magnitude less than Continuum}

Almost decomposable measure spaces are semilocalizable, \ref{ad.implies.semiloc} but the converse does not hold.
A classical counter-example (in case of semifinite measure spaces) is due to \textsc{D.H. Fremlin} \cite[216E]{FREMLIN.II}.
However if the corresponding quotient Boolean algebra is <<not too large>> the converse holds.
In case of semifinite measure spaces this is due to {\sc E. J. McShane} \cite{MCS.62}.
Here we deal with the non semifinite case, \ref{mcshane}.

\begin{Empty}
\label{61}
A measure space $(X,\calA,\mu)$ is called {\em locally determined} whenever the following holds:
\begin{equation*}
\forall A \in \calP(X) : \left[ \forall F \in \calA^f : A \cap F \in \calA \right] \Rightarrow A \in \calA 
\end{equation*}
where as usual 
\begin{equation*}
\calA^f = \calA \cap \{ A : \mu(A) < \infty \} \,.
\end{equation*}
\end{Empty}

 \begin{Lemma}
 \label{62}
Let $\phi$ be an outer measure on a set $X$ and assume that $\phi$ has measurable hulls, i.e.
\begin{equation*}
\left( \forall S \in \calP(X) \right) \left( \exists A \in \calA_\phi \right) : S \subset A \text{ and } \phi(S) = \phi(A) \,. 
\end{equation*}
It follows that $(X,\calA_\phi,\phi)$ is locally determined.
\end{Lemma}

\begin{proof}
Let $A \in \calP(X)$ and assume that $A \cap F$ is $\phi$ measurable whenever $F$ is $\phi$ measurable and $\phi(F) < \infty$. 
We ought to show that $A$ is $\phi$ measurable.
It suffices to establish that 
\begin{equation*}
\phi(S) \geq \phi(S \cap A ) + \phi(S \setminus A)
\end{equation*}
whenever $S \in \calP(X)$ and $\phi(S) < \infty$. 
Let $B \in \calA_\phi$ be a $\phi$ measurable hull of $S$.
Thus $B \in \calA_\phi^f$ so that $A \cap B \in \calA_\phi$ by assumption, and hence also $B \setminus A = B \setminus (A \cap B) \in \calA_\phi$. Therefore
\begin{equation*}
\phi(S) = \phi(B) = \phi(B \cap A) + \phi(B \setminus A) \geq \phi(S \cap A) + \phi(S \setminus A) 
\end{equation*}
and the proof is complete.
\end{proof}

\begin{Proposition}
\label{63}
Assume that $X$ is a Polish space and that $\phi$ is a Borel regular outer measure on $X$ (recall \ref{24}).
It follows that $(X,\calA_\phi,\phi)$ has magnitude (recall \ref{magnitude}) less than $\cac$ (the power of continuum).
\end{Proposition}

\begin{proof}
Let $\calE \subset \calA^f$ be as in \ref{magnitude}.
With each $E \in \calE$ we associate a Borel hull $B_E \in \calB(X)$ such that $E \subset B_E$ and $\phi(E)=\phi(B_E)$. 
Since $\phi(B_E) < \infty$ and both $E$ and $B_E$ are $\phi$ measurable, we infer that $\phi(B_E \setminus E)=0$. 
We now claim that if $E_1,E_2 \in \calE$ and $E_1 \neq E_2$ then $B_{E_1} \neq B_{E_2}$.
Indeed, assuming $E_1 \neq E_2$, we see that
\begin{equation*}
B_{E_1} \cap B_{E_2} \subset  (B_{E_1} \setminus E_1) \cup (E_1 \cap B_{E_2}) \subset (B_{E_1} \setminus E_1) \cup (E_1 \cap E_2) \cup ( B_{E_2} \setminus E_2) 
\end{equation*}
is $\phi$ negligible. 
Assuming if possible that $B_{E_1} = B_{E_2}$ it would ensue that $B_{E_1}$ is $\phi$ negligible, whence also $E_1$, a contradiction.
In other words the map $\calE \to \calB(X) : E \mapsto B_E$ is injective.
Since $\rmcard \calB(X) \leq \cac$ the proof is complete.
\end{proof}

\begin{Proposition}
\label{416}
Let $(X,\calA,\mu)$ be a measure space which is complete and locally determined, and let $\calE$ be such that
\begin{enumerate}
\item $\calE \subset \calA^f$;
\item $\calE$ is disjointed;
\item $( \forall A \in \calA^f \setminus \calN_\mu)(\exists E \in \calE) : A \cap E \not \in \calN_\mu$.
\end{enumerate}
It follows that $\calE$ is an almost decomposition of $(X,\calA,\mu)$. 
\end{Proposition}

\begin{proof}
We start by proving that condition (4) of \ref{almost.decomposition} is satisfied.
Let $A \in \calA^f$.
If $A \in \calN_\mu$ there is noting to prove, thus we henceforth assume that $\mu(A) > 0$.
Define
\begin{equation*}
\calE_A = \calE \cap \{ E : E \cap A \not \in \calN_\mu \} \,.
\end{equation*}
We first claim that $\calE_A$ is at most countable.
Indeed if $\calF \subset \calE_A$ is finite then 
\begin{equation*}
\sum_{E \in \calF} \mu(E \cap A) = \mu \left( (\cup \calF) \cap A \right) \leq \mu(A)
\end{equation*}
because $\calF$ is disjointed.
Letting $\calE_{A,n} = \calE_A \cap \{ E : \mu(E \cap A) \geq n^{-1} \}$ we infer that $\calE_A = \cup_{n=1}^\infty \calE_{A,n}$ and $\rmcard \calF \leq n \mu(A) < \infty$.
This completes the proof of the claim.
\par 
Define
\begin{equation*}
B = \cup_{E \in \calE_A} E \cap A
\end{equation*}
and notice that $B \in \calA$ because $\calE_A$ is at most countable.
Let $C = A \setminus B$.
Assume if possible that $\mu(C) > 0$.
Since $\mu(C) \leq \mu(A) < \infty$ it follows from hypothesis (3) that the exists $E \in \calE$ such that $\mu(E \cap C) > 0$. 
Thus $C \in \calE_A$ and in turn $E \cap C \subset B$ so that $E \cap C = E \cap (C \cap B) = \emptyset$ by the definition of $B$, a contradiction.
Thus indeed $\mu(C)=0$.
Finally
\begin{equation*}
\mu(A) = \mu(B) = \sum_{E \in \calE_A} \mu(E \cap A) = \sum_{E \in \calE} \mu(E \cap A) \,.
\end{equation*}
\par 
It remains to establish that condition (3) of \ref{almost.decomposition} holds.
Since $(X,\calA,\mu)$ is locally determined it suffices to show the following: If $A \in \calP(X)$ and $A \cap E \in \calA$ for every $E \in \calE$, then $A \cap F$ for every $F \in \calA^f$.
Fix $A \in \calP(X)$ that meets this condition.
Let $F \in \calA^f$.
We apply the preceding paragraph to $F$:
\begin{equation*}
F = \left( \cup_{E \in \calE_F} E \cap F \right) \cup N
\end{equation*}
for some $N \in \calN_\mu$.
Therefore
\begin{equation*}
A \cap F = \left( \cup_{E \in \calE_F} \left( (A \cap E) \cap F \right)\right) \cup (A \cap N) \,.
\end{equation*}
Now each $A \cap E \in \calA$ by assumption, thus also $A \cap E \cap F \in \calA$, $E \in \calE$.
Since $\calE_F$ is at most countable, the first term in the union of the right hand side above belongs to $\calA$.
Furthermore $A \cap N \in \calN_\mu \subset \calA$ because $(X,\calA,\mu)$ is complete.
Therefore $A \cap F \in \calA$.
Since $F \in \calA^f$ is arbitrary we are done.
\end{proof}

\begin{Proposition}
\label{mcshane}
Assume that a measure space $(X,\calA,\mu)$:
\begin{enumerate}
\item is complete;
\item is locally determined;
\item has magnitude less than $\cac$;
\item is semilocalizable.
\end{enumerate}
It follows that it is almost decomposable.
\end{Proposition}

\begin{proof}
According to \ref{existence.pu} applied with $\calI = \calA^f$ and $\calN = \calN_\mu$, there exists $\calE \subset \calA$ such that
\begin{enumerate}
\item[(a)] $\calE \cap \calN_\mu = \emptyset$;
\item[(b)] For every $E_1, E_2 \ in \calE$, if $E_1 \neq E_2$ then $E_1 \cap E_2 \in \calN_\mu$;
\item[(c)] For every $A \in \calA^f \setminus \calN_\mu$  there exists $E \in \calE$ such that $A \cap E \not \in \calN_\mu$.
\end{enumerate}
By hypothesis (3), $\rmcard \calE \leq \cac$ thus there exists an injective map $u : \calE \to ]0,1]$.
With each $E \in \calE$ we associate $g_E = u(E) \ind_E$ which is $\calA$ measurable, thus $(g_E)_{E \in \calE}$ is a family subordinated to $\calE$.
It is compatible relative to $(X,\calA,\calN_\mu[\calA^f])$ because if $E_1, E_2 \in \calE$ and $E_1 \cap E_2 \in \calN_\mu$ then 
\begin{equation*}
E_1 \cap E_2 \cap \{ g_{E_1} \neq g_{E_2} \} \subset E_1 \cap E_2 \in \calN_\mu \subset \calN_\mu[\calA^f] \,.
\end{equation*}
By hypothesis (4) it therefore admits a gluing $g$, i.e. an $\calA$ measurable function $g : X \to \R$ such that for every $E \in \calE$,
\begin{equation*}
E \cap \{ g \neq g_E \} \in \calN_\mu[\calA^f] \,.
\end{equation*}
Now for $E \in \calE$ we define
\begin{equation*}
G_E = E \cap g^{-1}\{u(E)\} \in \calA \,.
\end{equation*}
If $E_1 \neq E_2$ then $G_{E_1} \cap G_{E_2} \subset g^{-1}\{u(E_1)\} \cap g^{-1}\{u(E_1)\} = \emptyset$ because $u$ is injective. 
Accordingly, $\calG = \{ G_E : E \in \calE\}$ is disjointed.
Also, $\mu(G_E) \leq \mu(E) < \infty$ when $E \in \calE$, thus $\calG \subset \calA^f$. 
\par 
In order to conclude it remains only to establish that $\calG$ verifies condition (3) of \ref{416}.
Assume to the contrary that there exists $A \in \calA^f \setminus \calN_\mu$ such that $A \cap G_E \in \calN_\mu$ for all $E \in \calE$.
Given $E \in \calE$ we shall now show that $A \cap E \in \calN_\mu$, in contradiction with (c) above, thereby completing the proof.
Since $A \cap E = ( A \cap G_E) \cup ( A \cap (E \setminus G_E))$ and $A \cap G_E \in \calN_\mu$, it suffices to show that $A \cap (E \setminus G_E) \in \calN_\mu$.
Recall that $E \setminus G_E = E \cap \{ g \neq g_E \} \in \calN_\mu[\calA^f]$.
Since $A \in \calA^f$ we infer indeed that $A \cap (E \setminus G_E) \in \calN_\mu$.
\end{proof}

\begin{Corollary}
Assume that $(X,\calA,\mu)$ is a complete, locally determined measure space and has magnitude less than $\cac$. The following are equivalent.
\begin{enumerate}
\item $\Upsilon$ is surjective;
\item $(X,\calA,\mu)$ is semilocalizable;
\item $(X,\calA,\mu)$ admits an almost decomposition.
\end{enumerate}
\end{Corollary}

%\section{Local Lower and Upper Densities}

\section{Hausdorff Measures in Complete Separable Metric Spaces}

\begin{Theorem}
\label{71}
Let $X$ be a complete separable metric space and $0 < d < \infty$. 
For the measure space $(X,\calA_{\calH^d},\calH^d)$ the following are equivalent.
\begin{enumerate}
\item The canonical map $\Upsilon : \bL_\infty(X,\calA_{\calH^d},\calH^d) \to \bL_1(X,\calA_{\calH^d},\calH^d)^*$ is surjective;
\item $(X,\calA_{\calH^d},\calH^d)$ is semilocalizable;
\item $(X,\calA_{\calH^d},\calH^d)$ is almost decomposable.
\end{enumerate}
\end{Theorem}

\begin{proof}
That (1) and (2) be equivalent for any measure space was established in \ref{Riesz}, and that $(3) \Rightarrow (2)$ for any measure space was proved in \ref{ad.implies.semiloc}.
It remains to observe that $(2) \Rightarrow (3)$ under the present assumptions is a consequence of \ref{mcshane}.
The measure space $(X,\calA_{\calH^d},\calH^d)$ is clearly complete.
The outer measure $\calH^d$ being Borel regular, it follows from \ref{62} that  $(X,\calA_{\calH^d},\calH^d)$ is locally determined, and from \ref{63} and the separability assumption of $X$ that it has magnitude less than $\cac$.
\end{proof}

\section{An Abstract Condition for the Consistency of not being Semilocalizable}

\begin{Empty}[Cardinals related to $\sigma$-ideals]
\label{number.7}
Let $X$ be a set and let $\calN \subset \calP(X)$ be a $\sigma$-ideal in $\calP(X)$.
We recall the following cardinals associated with $\calN$:
\begin{equation*}
\begin{split}
\sfnon(\calN) & = \min \{ \rmcard S : S \subset X \text{ and } S \not \in \calN \} \\
\sfcov(\calN) & = \min \{ \rmcard \calC : \calC \subset \calN \text{ and } X = \cup \calC \} \,.
\end{split}
\end{equation*}
Letting $\calL^1$ denote the restriction of the Lebesgue {\it outer measure} to the interval $[0,1]$ we consider the corresponding cardinals $\sfnon(\calN_{\calL^1})$ and $\sfcov(\calN_{\calL^1})$.
These are part of the so-called Cicho\'n diagram, see \cite[522]{FREMLIN.V.1}.
Below we will use the fact that the strict inequality $\sfnon(\calN_{\calL^1}) < \sfcov(\calN_{\calL^1})$ is consistent with $\mathsf{ZFC}$, see \cite[Chapter 7]{BARTOSZYNSKI.JUDAH} or \cite[552H and 552G]{FREMLIN.V.2}.
\end{Empty}

\begin{Lemma}
\label{number.8}
Let $X$ be a Polish space and let $\mu$ be a diffuse probability measure defined on $\calB(X)$. 
It follows that:
\begin{enumerate}
\item $\sfnon(\calN_{\bar{\mu}}) = \sfnon(\calN_{\calL^1})$;
\item $\sfcov(\calN_{\bar{\mu}}) = \sfcov(\calN_{\calL^1})$.
\end{enumerate}
\end{Lemma}

\begin{proof}
Since $\mu$ is diffuse and nonzero, $X$ is uncountable and therefore the Kuratowski Isomorphism Theorem applies, \cite[3.4.23]{SRIVASTAVA}: There exists a bijection $f : X \to [0,1]$ such that both $f$ and $f^{-1}$ are Borel measurable, and $f_*\mu = \lambda$ where $\lambda = \calL^1|_{\calB([0,1])}$.
We claim that:
\begin{equation*}
\text{{\it For every} } S \subset X : \bar{\mu}(S) = 0 \text{ {\it if and only if} } \calL^1(f(S))=0 \,.
\end{equation*}
Assume that $\bar{\mu}(S)=0$.
Since $\bar{\mu}$ is Borel regular, \ref{26} there exists $\calB(X) \ni B \supset S$ such that $\mu(B)=0$.
As $f(B)$ is Borel one has $\calL^1(f(B))=\lambda(f(B))=(f_*\mu)(f(B))=\mu(B)=0$ and therefore $\calL^1(f(S))=0$ because $f(S) \subset f(B)$.
The other way round one argues similarly, referring to the Borel regularity of $\calL^1$.
\par 
We now prove (1).
Assume $S \subset [0,1]$ and $S \not \in \calN_{\calL^1}$, i.e. $\calL^1(S) > 0$. 
It follows from the claim above claim that $f^{-1}(S) \not \in \calN_{\bar{\mu}}$.
Since $\rmcard S = \rmcard f^{-1}(S)$ we infer that $\sfnon(\calN_{\bar{\mu}}) \leq \sfnon(\calN_{\calL^1})$.
The reverse inequality is proved in a similar fashion.
\par 
Let $(N_i)_{i \in I} \subset \calN_{\calL^1}$ be such that $[0,1] = \cup_{i \in I} N_i$.
It follows from the claim above that $(f^{-1}(N_i))_{i \in I} \subset \calN_{\bar{\mu}}$, and clearly $X = \cup_{i \in I} f^{-1}(N_i)$, thus $\sfcov(\calN_{\bar{\mu}}) \leq \sfcov(\calN_{\calL^1})$.
The reverse inequality follows similarly and the proof of (2) is complete.
\end{proof}

\begin{Theorem}[$\mathsf{ZFC} + \sfnon(\calN_{\calL^1}) < \sfcov(\calN_{\calL^1})$]
\label{abstract.theorem}
Assume that $(X,\calA,\calN)$ is a measurable space with negligibles and that $(S,\calB(S),\sigma)$, $(T,\calB(T),\tau)$ are diffuse probability spaces with $S$ and $T$ being Polish.
Furthermore assume the existence of maps $S \to \calA : s \mapsto V_s$ and $T \to \calA : t \mapsto H_t$ with the following properties.
\begin{enumerate}
\item For every $s \in S$ and every $t \in T$ one has $\emptyset \neq V_s \cap H_t \in \calN$;
\item For every $Z \in \calA$ one has:
\begin{enumerate}
\item For every $s \in S$ if $V_s \cap Z \in \calN$ then 
\begin{equation*}
\overline{\tau} ( T \cap \{ t : H_t \cap V_s \cap Z \neq \emptyset \} ) = 0 \,;
\end{equation*}
\item For every $t \in T$ if $H_t \cap Z \in \calN$ then 
\begin{equation*}
\overline{\sigma } ( S \cap \{ s : V_s \cap H_t \cap Z \neq \emptyset \} ) = 0 \,.
\end{equation*}
\end{enumerate}
\end{enumerate}
Under the consistent assumption that $\sfnon(\calN_{\calL^1}) < \sfcov(\calN_{\calL^1})$ it follows that $(X,\calA,\calN)$ is not localizable.
\end{Theorem}

\begin{proof}
Proceeding toward a contradiction, assume if possible that $(X,\calA,\calN)$ is localizable.
The family $\calE = \{ V_s : s \in S \} \subset \calA$ would then admit an $\calN$ essential supremum, say $A \in \calA$.
Thus,
\begin{enumerate}
\item[(A)] For every $s \in S$: $V_s \setminus A \in \calN$;
\item[(B)] For every $t \in T$: $H_t \cap A \in \calN$.
\end{enumerate}
Condition (A) readily follows from the definition of an essential supremum whereas condition (B) is established in the following manner.
Fix $t \in T$ and consider the set $B = A \setminus H_t$ and observe that for every $s \in S$ one has $V_s \setminus B = (V_s \setminus A) \cup (V_s \cap H_t) \in \calN$ since the first term is a member of $\calN$ by (A), and the second by hypothesis (1).
As $s$ is arbitrary, $B$ is an $\calN$ essential upper bound of $\calE$. 
Therefore $A \setminus B \in \calN$ and it remains to notice that $A \setminus B = A \cap H_t$. 
\par 
With each $s \in S$ we now associate the set 
\begin{equation*}
T_s = T \cap \{ t : H_t \cap V_s \cap A^c \neq \emptyset \} \,.
\end{equation*}
Since $V_s \cap A^c \in \calN$ by (A), our hypothesis (2)(a) implies  that $\overline{\tau}(T_s)=0$.
Now let $E \subset S$ be such that $\overline{\sigma}(E) > 0$ and 
\begin{equation*}
\rmcard E = \sfnon(\calN_{\bar{\sigma}}) = \sfnon(\calN_{\calL^1}) < \sfcov(\calN_{\calL^1}) = \sfcov(\calN_{\bar{\tau}}) \,,
\end{equation*}
where the second and last equalities follow from \ref{number.8} and the strict inequality is our consistent assumption.
We next define $F = \cup_{s \in E} T_s \subset T$.
Since each $T_s \in \calN_{\bar{\tau}}$ and $\rmcard E < \sfcov(\calN_{\bar{\tau}})$ it ensues that $F \neq T$.
Pick $t \in T \setminus F$.
Thus for each $s \in E$ one has $t \not \in T_s$, i.e. $H_t \cap V_s \cap A^c = \emptyset$ and in turn (since $V_s \cap H_t \neq \emptyset$ by assumption (1)) $H_t \cap V_s \cap A \neq \emptyset$. 
Accordingly, 
\begin{equation*}
E \subset S \cap \{ s : V_s \cap H_t \cap A \neq \emptyset \} \,.
\end{equation*}
Yet $H_t \cap A \in \calN$ for every $t \in T$ by (B), thus $\overline{\sigma}(E)=0$ by assumption (2)(b), a contradiction.
\end{proof}

\begin{Empty}
\label{84}
One checks (in a similar way as in \ref{85} below) that \ref{abstract.theorem} applies to the measurable space with negligibles $(X,\calA,\calN)$ where $X = [0,1] \times [0,1]$, $\calA$ is a $\sigma$ algebra of subsets of $X$ containing the Borel subsets of $X$, and $\calN$ consists of those members of $\calA$ that are $\calH^1$ negligible.
The hypotheses are met with both $(S,\calB(S),\sigma)$ and $(T,\calB(T),\tau)$ being $([0,1],\calB([0,1]),\calL^1|_{\calB([0,1])})$ and $V_s = \{s\} \times [0,1]$, $H_t = [0,1] \times \{t\}$, $s,t \in [0,1]$.
Thus $V_s \cap H_t = \{(s,t)\}$ and the condition $V_s \cap H_t \cap Z \neq \emptyset$ is equivalent to $(s,t) \in Z$, and therefore (2)(a) (resp. (2)(b)) of \ref{abstract.theorem} holds since the corresponding slice of $Z$ is assumed to be $\calH^1$ negligible, and the projection on the second (resp. first) axis contracts $\calH^1$ measure.
The case when $\calA$ consists of those $\calH^1$ measurable subsets of $X$ was proved in \cite{DEP.98}.
The notion of a measurable space with negligibles makes it a possibility to dispense altogether with a measure being defined on $\calA$ and, consequently allows for our slightly more general statement here.
The point being that the nature of the statement does not involve a measure. 
See also \ref{CR.2}(Q1).
\end{Empty}

\begin{Theorem}
\label{85}
Let $0 < d < 1$ and let $C_d \subset [0,1]$ be a self-similar Cantor set of Hausdorff dimension $d$ described in \cite[4.10]{MATTILA}.
The measure space $(C_d \times C_d, \calA_{\calH^d},\calH^d)$ is consistently not semilocalizable.
\end{Theorem}

\begin{proof}
We let $X = C_d \times C_d$, we let $\phi$ be the Hausdorff $\calH^d$ measure restricted to $X$ and $\calA = \calA_\phi$.
We also put $\calN = \calN_\phi[\calA_\phi^f]$ and we aim at checking that \ref{abstract.theorem}applies to the measurable space with negligibles $(X,\calA,\calN)$.
To this end we consider the probability spaces $(S,\calB(S),\sigma)$ and $(T,\calB(T),\tau)$ both equal to $(C_d,\calA_{\calH^d \hel C_d},\calH^d \hel C_d)$.
We further define $V_s = \{s\} \times C$ and $H_t = C \times \{t \}$, $s,t \in C$.
These belong to $\calA$ because they are Borel and $\phi$ is Borel regular.
For each $s,t \in C$, $\emptyset \neq V_s  \cap H_t = \{(s,t)\} \in \calN_\phi \subset \calN$ thus hypothesis (1) of \ref{abstract.theorem} is verified.
Now let $Z \subset X_d$ be $\calH^d$ measurable, $s \in C$, and assume that $V_s \cap Z \in \calN = \calN_\phi[\calA_\phi^f]$.
Since $\calH^d(V_s) = \calH^d(C) < \infty$ and $V_s \cap Z = V_s \cap (V_s \cap Z)$ we instantly infer that $\calH^d(V_s \cap Z)=0$.
Furthermore,
\begin{equation*}
C \cap \{ t : H_t \cap V_s \cap Z \neq \emptyset \} = C \cap \{ t : (s,t) \in V_s \cap Z  \} = \pi_1(V_s \cap Z)
\end{equation*}
and in turn
\begin{equation*}
\calH^d \left( C \cap \{ t : H_t \cap V_s \cap Z \neq \emptyset \}\right) \leq (\rmLip \pi_1)^d \calH^d(V_s \cap Z) = 0 \,.
\end{equation*}
This proves that condition (2)(b) of \ref{abstract.theorem} is satisfied in the present case.
Part (b) is checked in a similar fashion.
\end{proof}

\section{Purely Unrectifiable Example}
\label{section.pue}

\begin{Empty}[The purely unrectifiable set $X$]
\label{pue.1}
We are given a sequence $(\lambda_k)_{k \in \N}$ of positive real numbers such that $\lambda_0 = 1$ and $0 < \lambda_k < \frac{1}{2} \lambda_{k-1}$ for every $k \geq 1$.
We will define inductively a sequence $(\calX_k)_{k \in \N}$ of sets of squares in $\R^2$.
We start with $X_{0,0} = [0,1] \times [0,1]$ and $\calX_0=\{X_{0,0}\}$.
We let $\calX_k$ consists of $4^k$ closed squares: It contains four subsquares of each $S \in \calX_{k-1}$, each having a vertex in common with $S$ and sidelength $\lambda_k$.
We let 
\begin{equation*}
X = \cap_{k=1}^\infty \cup \calX_k \,.
\end{equation*}
It clearly follows from the definitions that $\cup \calX_k \subset \cup \calX_{k-1}$ and that $\calX_k$ consists of $4^k$ pairwise disjoint nonempty compact sets.
Consequently $X$ is (topologically) a Cantor space.
\par 
One next defines $C \subset [0,1]$ as $C = \cap_{k=0}^\infty \cup \calC_k$, where $(\calC_k)_{k \in \N}$ is defined inductively as follows.
$\calC_0=\{[0,1]\}$.
We let $\calC_k$ consists of $2^k$ closed intervals: It contains two subintervals of each $I \in \calC_{k-1}$, each having an endpoint in common with $I$ and length $\lambda_k$.
\begin{enumerate}
\item {\it $X = C \times C$ and if $\calL^1(C)=0$ then $X$ is purely $(\calH^1,1)$ unrectifiable.}
\end{enumerate}
The first assertion follows from the observation that $S \in \calX_k$ if and only if $S = I \times J$ for some $I,J \in \calC_k$.
The second assertion follows from \cite[18.10(4)]{MATTILA}.
\end{Empty}

\begin{Empty}[Numbering of $\calX_k$ and $\calI_k$]
\label{pue.2}
We observe that each $S \in \calX_k$ is contained in a unique $T \in \calX_{k-1}$.
It will be convenient to number $\calX_k = \{ X_{k,j} : j=0,\ldots,4^k-1 \}$ in such a way that $X_{k,j} \subset X_{k-1,\lfloor j/4 \rfloor}$, $k \in \N^*$, $j=0,\ldots,4^k-1$.
This is readily feasible.
\par 
We next consider the sequence $(\calI_k)_{k \in \N}$ of subsets of $[0,1]$ defined as follows.
We put $I_{0,0}=[0,1]$ and $\calI_0=\{I_{0,0}\}$, and we let $\calI_k$ consist of $4^k$ nonoverlapping compact subintervals of $[0,1]$, each of length $4^{-k}$, such that $[0,1] = \cup \calI_k$.
We notice that each $I \in \calI_k$ is contained in a unique $J \in \calI_{k-1}$.
We choose a numbering of $\calI_k = \{ I_{k,\ell} : \ell=0,\ldots,4^k-1 \}$ in such a way that $I_{k,\ell} \subset I_{k-1,\lfloor \ell/4 \rfloor}$, $k \in \N^*$, $\ell=0,\ldots,4^k-1$.
\par 
Given two integers $j,j' \in \N$ we say that $j'$ is a {\em daughter} of $j$ if $j = \lfloor j'/4 \rfloor$.
We say that a sequence $(j_k)_{k \in \N}$ of nonnegative integers is a {\em lineage} if $j_{k-1}$ is a daughter of $j_k$ for every $k \geq 1$.
The following now follows from our choice of numbering.
\begin{enumerate}
\item {\it 
Let $(j_k)_{k \in \N}$ be a sequence of nonnegative integers. The sequence $(X_{k,j_k})_{k \in \N}$ (resp. $(I_{k,j_k})_{k \in \N}$) is decreasing if and only if $(j_k)_{k \in \N}$ is a lineage.
}
\end{enumerate}
\end{Empty}

\begin{Empty}[Coding]
\label{pue.3}
Here we will define functions $j : \N \times X \to \N$ and $\ell : \N \times Y \to \N$ where $Y \subset [0,1]$ is to be described momentarily.
Given $x \in X$ and $k \in \N$, there exists a unique $j(k,x) \in \{0,\ldots,4^{k}-1\}$ such that $x \in X_{k,j(k,x)}$.
This is because the family $\calX_k$ is disjointed.
Furthermore $(X_{k,j(k,x)})_{k \in \N}$ is decreasing, i.e. $(j(k,x))_{k \in \N}$ is a lineage.
\par 
If $y \in [0,1]$ and $k \in \N$ there does not necessarily exist a {\it unique} $\ell \in \{0,\ldots,4^k -1  \}$ such that $y \in I_{k,\ell}$. 
\begin{enumerate}
\item {\it 
For every $y \in [0,1]$ and every $k \in \N$ there exists a unique $\ell \in \{0,\ldots,4^k - 1 \}$ such that $y \in I_{k,\ell}$ if and only if $y \not \in D_k$ where $D_k = \{ j.4^{-k} : j=1,\ldots,4^k - 1 \}$.
}
\end{enumerate}
If instead $y \in \{ j.4^{-k} : j=1,\ldots,4^k -1 \}$ then there are exactly two such $\ell$'s
\begin{enumerate}
\item[(2)] {\it 
Assume $y \in [0,1]$. There are at most two lineages $(\ell_k)_{k \in \N}$ such that $y \in I_{k,\ell_k}$ for every $k \in \N$. 
}
\end{enumerate}
In order to prove this, assume $(\ell_k)_{k \in \N}$, $(\ell'_k)_{k \in \N}$ and $(\ell''_k)_{k \in \N}$ are three lineages, at least two of which are distinct, such that $y \in I_{k,\ell_k} \cap I_{k,\ell'_k} \cap I_{k,\ell''_k}$ for every $k \in \N$.
Let $k_0$ be the least integer such that $\{\ell_{k_0},\ell'_{k_0},\ell''_{k_0}\}$ is not a singleton.
Renaming the sequences if necessary we may assume $\ell_{k_0}\neq\ell'_{k_0}$.
Since any three distinct members of $\calI_{k_0}$ have empty intersection it follows that either $\ell''_{k_0}=\ell_{k_0}$ or $\ell''_{k_0}=\ell'_{k_0}$.
Renaming again the sequences if necessary we may assume the first case occurs.
It remains to observe, by induction on $m$ that $\ell''_{k_0+m} = \ell_{k_0+m}$, $m \in \N$.
This is because of the two members of $\calI_{k_0+m}$ that contain $y$, only one is contained in $I_{\ell''_{k_0+m-1}}$.
\par 
To close this number we define $D = \cup_{k \in \N^*} D_k$ and $Y= [0,1] \setminus D$.
Thus for every $y \in Y$ and every $k \in \N$ there exists a unique $\ell(k,y) \in \{0,\ldots,4^k - 1 \}$ such that $y \in I_{k,\ell(k,y)}$.
It follows that $(I_{k,\ell(k,y)})_{k \in \N}$ is decreasing, hence $(\ell(k,y))_{k \in \N}$ is a lineage.
\end{Empty}

\begin{Proposition}
\label{pue.4}
There exists a Borel isomorphism $f : X \to [0,1]$ and a countable set $E \subset X$ with the following properties.
\begin{enumerate}
\item For every $k \in \N$ and every $S \in \calX_k$ there exists $I \in \calI_k$ such that $f(S\setminus E) \subset I$;
\item For every $k \in \N$ and every $I \in \calI_k$ there exists $S \in \calX_k$ such that $f^{-1}(I\setminus D) \subset S$;
\item $f(E)=D$.
\end{enumerate}
\end{Proposition}

\begin{proof}
We start by defining a map $g : X \to [0,1]$.
Given $x \in X$ we consider the lineage $(j(k,x))_{k \in \N}$ defined in \ref{pue.2}.
It follows from \ref{pue.2}(1) that $(I_{k,j(k,x)})_{k \in \N}$ is a decreasing sequence of compact intervals, whose $k^{th}$ term has length $4^{-k}$.
Accordingly there exists $g(x) \in [0,1]$ such that
\begin{equation*}
\{g(x)\} = \cap_{k \in \N} I_{k,j(k,x)} \,.
\end{equation*}
Now for each $k \in \N$ and $j \in \{0,\ldots,4^k -1 \}$ we pick $y_{k,j} \in I_{k,j}$ arbitrarily, and we observe that
\begin{equation*}
g(x) = \lim_k \sum_{j=0}^{4^k-1} y_{k,j} \ind_{X_{k,j}}(x) \,,
\end{equation*} 
$x \in X$.
This shows that $g$ is Borel measurable.
\par 
Letting $D \subset [0,1]$ be defined as in \ref{pue.3} and $E = g^{-1}(D)$ we infer that $g|_{X \setminus E}$ is injective.
Suppose indeed that $x,x' \in X$ are such that $g(x)=g(x') \not \in D$.
It follows from \ref{pue.3}(1) and the definition of $g$ that $j(k,x) = j(k,x')$, hence $\|x-x'\| \leq \rmdiam X_{j,k(k,x)} =4^{-k} \sqrt{2}$, for all $k \in \N$, thus $x=x'$.
We further claim that $g(X \setminus E)=[0,1] \setminus D$.
If indeed $y \in [0,1] \setminus D$ we consider the lineage $(\ell(k,y))_{k \in \N}$ defined in \ref{pue.3}, so that $(X_{k,\ell(k,y)})_{k \in \N}$ is a decreasing sequence according to \ref{pue.2}(1) and hence there exists $h(y) \in X$ such that
\begin{equation*}
\{h(y)\} = \cap_{k \in \N} X_{k,\ell(k,y)} \,.
\end{equation*}
Upon observing that $j(k,g(y))=\ell(k,y)$ it follows from the definition of $g$ that $g(h(y))=y$.
By definition of $E$, $h(y) \in X \setminus E$.
In other words $h$ is the inverse of the bijection $X \setminus E \to [0,1] \setminus D : x \mapsto g(x)$.
Picking $x_{k,j} \in X_{k,j}$ arbitrarily, $k \in \N$, $j=0,\ldots,4^k-1$, we note that
\begin{equation*}
h(y) = \lim_k \sum_{j=0}^{4^k-1} x_{k,j} \ind_{I_{k,j}}(y) \,,
\end{equation*}
$y \in [0,1] \setminus D$, thereby showing that $h$ is Borel measurable.
\par 
Next we infer from \ref{pue.3}(2) and the definition of $g$ that $g^{-1}\{y\}$ contains at most two members, $y \in D$. 
Since $D$ is countable it follows that so is $E = g^{-1}(D)$.
Choose arbitrarily a bijection $\vphi : E \to D$ and define $f : X \to [0,1]$ by 
\begin{equation*}
f : X \to [0,1] : x \mapsto \begin{cases}
g(x) & \text{ if } x \not \in E \\
\vphi(x) & \text{ if } x \in E \,.
\end{cases}
\end{equation*}
It is now obvious that $f$ is a bijection, and that both $f|_{X \setminus E}$ and $f|_E$ are Borel isomorphisms.
Thus $f$ itself is a Borel isomorphism.
\par 
Let $k_0 \in \N$ and $S \in \calX_{k_0}$.
It readily follows from the definition of $f$ that $f(S \setminus E) \subset g(S)$.
Let $j_0$ be such that $S = X_{k_0,j_0}$.
If $x \in S$ then $j(k_0,x)=j_0$ so that, by definition of $f$, $g(x) \in I_{k_0,j_0}$. 
This proves (1).
Similarly let $I = I_{k_0,j_0} \in \calI_{k_0}$.
Clearly $f^{-1}(I \setminus D) = h(I \setminus D)$.
If $y \in I \setminus D$ then $\ell(k_0,y) = j_0$ whence $h(y) \in X_{k_0,j_0}$.
This proves (2). 
\end{proof}

\begin{Empty}[Choice of a Cantor set $C_d$ and corresponding $X_d = C_d \times C_d$]
\label{pue.5}
Recall the construction of $X$ in \ref{pue.1}.
For the remaining part of this section $C$ will be a self-similar Cantor set of Hausdorff dimension $0 < d \leq \frac{\log 2}{\log 3}$.
Thus we let the family $\calC_k$ consists of $2^k$ members of length $\lambda_k = \lambda^k$ where $d = \frac{\log 2}{\log \lambda^{-1}}$, i.e. $0 < \lambda \leq \frac{1}{3}$ (see e.g. \cite[4.10]{MATTILA} or \cite[2.10.28]{GMT}).
We choose our notation to reminisce about this choice by letting $C_d$ denote the corresponding Cantor set, and $X_d = C_d \times C_d$.
\end{Empty}

\begin{Proposition}
\label{pue.6}
For every $Z \subset X_d$ one has 
\begin{equation*}
\left( \frac{1}{2} \right)^{1+\frac{d}{2}} \calH^d(Z) \leq \calH^\oh(f(Z)) \leq \left( \frac{2}{\lambda}\right)^d \calH^d(Z) \,.
\end{equation*}
In particular
$\calH^d(Z)=0$ if and only if $\calH^{\frac{1}{2}}(f(Z))=0$.
\end{Proposition}

\begin{proof}
We start by observing that for every $k \in \N$ and every $S \in \calX_k$, $I \in \calI_k$ one has
\begin{equation}
\label{eq.5}
\sqrt{\rmdiam I} = \sqrt{4^{-k}} = 2^{-k} = \lambda^{kd} = \left(\frac{1}{\sqrt{2}} \rmdiam S \right)^d
\end{equation}
because $2^k (\lambda^k)^d=1$ by our choice of $\lambda$ and $d$, \ref{pue.5}.
\par 
Let $Z \subset X_d$ and let $\veps > 0$.
In order to prove the right hand inequality there is no restriction to assume that $\calH^d(Z) < \infty$.
There exists a finite or countable covering $(U_i)_{i \in I}$ of $Z$ in $X$ such that $\sum_i (\rmdiam U_i)^d < \veps + \calH^d_{(\veps)}(Z)$ and $\rmdiam U_i < \veps$ for every $i \in I$.
Note we may assume each $U_i$ is open and nonempty.
Abbreviate $\delta_i = \rmdiam U_i$, $i \in I$.
Choosing $x_i \in U_i$ and letting $U'_i = \bU(x_i,\delta_i) \cap X$ we see that $\rmdiam U'_i \leq 2 \delta_i$, $i \in I$.
Choose $k_i \in \N$ such that $\lambda^{-(k_i+1)} \leq \rmdiam U'_i < \lambda^{-k_i}$.
Thus $U'_i$ intersects some $X_{k_i,j_i}$, $j_i \in \{0,\ldots,4^{k_i}-1\}$, and since $\lambda \leq \frac{1}{3}$ it intersects only one of them.
Therefore $U'_i \subset X_{k_i,j_i}$.
Notice that $\rmdiam X_{k_i,j_i} = \lambda^{k_i} \sqrt{2} \leq 2\sqrt{2}\lambda^{-1} \delta_i$.
Now $Z \subset \cup_{i \in I} \subset \cup_{i \in I} X_{k_i,j_i}$ thus also $Z \setminus E \subset  \cup_{i \in I} (X_{k_i,j_i} \setminus E)$ and it follows from \ref{pue.4}(1) that $f(Z \setminus E) \subset \cup_{i \in I} I_{k_i,j'_i}$ for some integers $j'_i \in \N$, $i \in I$. 
In turn \eqref{eq.5} implies that
\begin{equation*}
\calH^{\frac{1}{2}}_{(\eta_\veps)}(f(Z \setminus E)) \leq \sum_{i \in I} \sqrt{ \rmdiam I_{k_i,j'_i}} \leq \left( \frac{2}{\lambda}\right)^d \sum_{i \in I} ( \rmdiam U_i)^d \leq \left( \frac{2}{\lambda}\right)^d \left(\veps+\calH^d_{(\veps)}(Z) \right) \,,
\end{equation*}
where $\eta_\veps = 2^{3d-1}\lambda^{-2d}\veps^{2d}$.
Since $\veps > 0$ is arbitrary we infer that $\calH^{\frac{1}{2}}(f(Z\setminus E)) \leq 2^d \lambda^{-d} \calH^d(Z)$. 
As $f(Z) \subset f(Z \setminus E) \cup f(E)$ and $f(E)=D$ is countable, $\calH^d(f(E))=0$ and the right hand inequality follows.
\par 
Let $A \subset [0,1]$ and let $\veps > 0$.
In order to prove the left hand inequality we may of course suppose that $\calH^\oh(A) < \infty$.
There exists a covering $(U_i)_{i \in I}$ of $A$ in $[0,1]$ such that $\sum_{i \in I} \sqrt{ \rmdiam U_i} < \veps + \calH^\oh_{(\veps)}(A)$ and $\rmdiam U_i < \veps$ for every $i \in I$.
There is no restriction to assume that each $U_i$ is a nondegenerate interval.
We abbreviate $\delta_i = \rmdiam U_i$ and we let $k_i \in \N$ be such that $4^{-(k_i+1)} \leq \delta_i < 4^{k_i}$.
Notice that $U_i$ intersects at most two members of $\calI_{k_i}$, i.e. there exist $J_{i,1},J_{i,2} \in \calI_{k_i}$ such that $U_i \subset J_{i,1} \cup J_{i,2}$.
Furthermore $\rmdiam J_{i,q} \leq 4 \rmdiam U_i$, $i \in I$, $q=1,2$.
Now observe that 
\begin{equation*}
f^{-1}(A \setminus D) \subset \cup_{i \in I} \left( f^{-1}(J_{i,1} \setminus D) \cup f^{-1} (J_{i,2} \setminus D) \right) \subset \cup_{i \in I} \left( X_{k_i,j_{i,1}} \cup X_{k_i,j_{i,1}}\right)
\end{equation*}
for some integers $j_{i,1},j_{i,2} \in \{0,\ldots,4^{k_i}-1\}$, according to \ref{pue.4}(2).
Therefore it follows from \eqref{eq.5} that
\begin{multline*}
\calH^d_{(\eta_\veps)}(f^{-1}(A \setminus D))  \leq \sum_{i \in I} \left( \left( \rmdiam X_{k_i,j_{i,1}} \right)^d + \left( \rmdiam X_{k_i,j_{i,2}} \right)^d\right) \\
 \leq 2^{1 + \frac{d}{2}}\sum_{i \in I} \sqrt{\rmdiam U_i} \leq  2^{1 + \frac{d}{2}} \left( \veps + \calH^\oh_{(\veps)}(A) \right)\,,
\end{multline*}
where $\eta_\veps = 2^\frac{1}{d} (4\veps)^\frac{2}{d}$.
Since $\veps > 0$ is arbitrary, $\calH^d(f^{-1}(A \setminus D))=0$.
Finally $\calH^d(f^{-1}(A)) \leq 2^{1+\frac{1}{d}}\calH^\oh(A)$ because $f^{-1}(A) \subset f^{-1}(A \setminus D) \cup f^{-1}(D)$ and $f^{-1}(D)=E$ is countable whence $\calH^d(f^{-1}(D))=0$.
\end{proof}

\begin{Empty}
\label{pue.7}
We say that a measure space $(X,\calA,\mu)$ is {\em undecidably semilocalizable} if the proposition <<$(X,\calA,\mu)$ is semilocalizable>> is undecidable in $\mathsf{ZFC}$.
In case $X$ is a complete separable metric space and $0 < d < \infty$, the measure space $(X,\calA_{\calH^d},\calH^d)$ is undecidably semilocalizable if and only if the proposition <<$(X,\calA_{\calH^d},\calH^d)$ is almost decomposable>> is undecidable in $\mathsf{ZFC}$.
This is a consequence of \ref{71}.
\end{Empty}

\begin{Theorem}
\label{pue.8}
The measure space $\left([0,1],\calA_{\calH^{1/2}},\calH^\oh\right)$ is undecidably semilocalizable.
\end{Theorem}

\begin{proof}
It is consistently semilocalizable according to \ref{CH.implies.ad}.
Fix $0 < d \leq \frac{\log 2}{\log 3}$ arbitrarily.
Let $f : X_d \to [0,1]$ be as before and define a $\sigma$-algebra of subsets of $X_d$ by the formula
\begin{equation*}
\calA = \calP(X_d) \cap \left\{ f^{-1}(A) : A \in \calA_{\calH^\oh} \right\} \,.
\end{equation*}
We claim that $\calB(X_d) \subset \calA$.
Indeed if $B \subset X_d$ is Borel then so is $f(B)$, according to \ref{pue.4}.
Thus $f(B) \in \calA_{\calH^{1/2}}$ and in turn $B = f^{-1}(f(B)) \in \calA$.
We also define
\begin{equation*}
\calN = \calP(X_d) \cap \left\{ f^{-1}(N) : N \in \calN_{\calH^{1/2}}\left[\calA_{\calH^{1/2}}^f\right] \right\} \,.
\end{equation*}
It ensues from their construction that the measurable spaces with negligibles $(X_d,\calA,\calN)$ and $\left([0,1],\calA_{\calH^{1/2}},\calN_{\calH^{1/2}}\left[\calA_{\calH^{1/2}}^f\right] \right)$ are isomorphic in the category $\sfMSN$.
According to \ref{last.lemma} one is localizable if and only if the other one is.
Reasoning as in the proof of \ref{85} we will now show that the former is localizable.
First we notice that $V_s = \{s\} \times C_d$ and $H_t = C_d \times \{t\}$, $s,t \in C_d$, indeed belong to $\calA$ because $\calB(X_d) \subset \calA$.
Let $Z \in \calA$ and $s \in C_d$ be such that $V_s \cap Z \in \calN$.
This means that $f(V_s \cap Z) \in \calN_{\calH^{1/2}}\left[\calA_{\calH^{1/2}}^f\right]$.
Since $f(V_s \cap Z) = f(V_s) \cap f(Z)$ and $f(V_s) \in \calA_{\calH^{1/2}}^f$ according to \ref{pue.6} we infer that $\calH^\oh(f(V_s \cap Z))=0$ and in turn $\calH^d-V_s \cap Z)=0$ again thanks to \ref{pue.6}.
Reasoning as in \ref{85} we conclude that $\calH^d ( C \cap \{ t : H_t  \cap V_s \cap Z \neq \emptyset \}) = 0$.
Similarly the condition (2)(b) of \ref{abstract.theorem} holds as well and the proof is complete.  
\end{proof}

\section{Purely Rectifiable Example}
\label{example}

\begin{Empty}[A Cantor set]
\label{number.1}
We let $\{0,1\}^{\N^*}$ be the Cantor space equipped with its usual topology and its usual Borel, probability, product measure $\lambda$.
For each $j \in \N^*$ we let $\calS_j$ be a collection of disjoint, compact subintervals of $[0,1]$, and we let $(\ell_j)_{j \in \N^*}$ be a sequence in $(0,1)$, with the following properties:
\begin{enumerate}
\item $\rmcard \calS_j = 2^j$;
\item For every $T \in \calS_j$ one has $\rmcard \calS_{j+1} \cap \{ S : S \subset T \} = 2$;
\item For every $S \in \calS_j$ one has $\calL^1(S)=\ell_j$.
\end{enumerate}
We then define $C = \cap_{j \in \N^*} \cup \calS_j$.
This way we can realize a set $C$ of any Hausdorff dimension $0 \leq d < 1$, see e.g. \cite[4.10 and 4.11]{MATTILA}.
We will be mostly interested in the case $d=0$.
In any case we will henceforth assume that $\calL^1(C)=0$.
\par 
For each $j \in \N^*$ we number the members of $\calS_j$ as $S_{j,0},\ldots,S_{j,2^j-1}$ in such a way that $\max S_{j,k} < \min S_{j,k+1}$, $k=0,\ldots,2^j-2$.
Thus $S_{j+1,2k} \cup S_{j+1,2k+1} \subset S_{j,k}$ for all $j \in \N^*$ and all $k=0,\ldots,2^j-1$.
Now given $\xi \in \{0,1\}^{\N^*}$ we define inductively $(k_\xi(j))_{j \in \N^*}$ as follows: $k_\xi(1)=\xi(1)$ and $k_\xi(j+1)=2k_\xi(j)+\xi(j+1)$.
In turn we define the usual coding of $C$,
\begin{equation*}
\vphi : \{0,1\}^{\N^*} \to C
\end{equation*}
by letting $\vphi(\xi)$ be the only point of $[0,1]$ such that
\begin{equation*}
\{ \vphi(\xi) \} = \cap_{j \in \N^*} S_{j,k_\xi(j)} \,.
\end{equation*}
Thus $\vphi$ is a homeomorphism.
\end{Empty}

\begin{Empty}[The measures $\mu$ and $\mu_j$]
\label{number.2}
We define a Borel probability $\mu$ measure on $[0,1]$ by the formula $\mu(B) = \lambda(\vphi^{-1}(B \cap C))$, $B \in \calB([0,1])$. 
For each $j \in \N^*$ we define a Borel probability measure $\mu_j$ on $[0,1]$ by the formula
\begin{equation*}
\mu_j = \left(\frac{1}{2^j \ell_j}\right) \calL^1 \hel (\cup \calS_j)  \,.
\end{equation*}
\end{Empty}

\begin{Lemma}
\label{number.3}
The sequence $(\mu_j)_{j \in \N^*}$ converges weakly* to $\mu$.
\end{Lemma}

\begin{proof}
First we let $S \in \calS_j$ for some $j \in \N^*$.
Observe that $\mu(S) = 2^{-j}$.
If $k \geq j$ then $\mu_k(S) = 2^{-k} \ell_k^{-1} \calL^1( S \cap \cup \calS_k)= 2^{-k}\ell_k^{-1} (2^{k-j} \ell_k) = 2^{-j}$.
In particular $\lim_k \mu_k(S) = \mu(S)$.
\par 
Next we let $U \subset [0,1]$ be relatively open.
There exists a disjointed sequence $(S_n)_{n \in \N}$ of members of $\cup_{j \in \N^*} \calS_j$ such that each $S_n \subset U$ and
\begin{equation*}
C \cap U = C \cap \left( \cup_{n \in \N} S_n \right) \,.
\end{equation*}
It suffices indeed to let $(S_n)_{n \in \N}$ be a numbering of $\calT = \cup_{j \in \N^*} \calT_j$ where $(\calT_j)_{j\in \N^*}$ is defined inductively as follows:
$\calT_1 = \calS_1 \cap \{ S : S \subset U \}$ and $\calT_{j+1} = \calS_{j+1} \cap \{ S : S \subset U \text{ and } S \cap \cup_{k=1}^j \cup \calT_k = \emptyset \}$.
Now given $\veps > 0$ there exists $N \in \N$ such that 
\begin{equation*}
\sum_{n \in \N} \mu(S_n) \leq \veps + \sum_{n=0}^N \mu(S_n) \,.
\end{equation*}
Furthermore,
\begin{multline*}
\mu(U) = \sum_{n \in \N} \mu(S_n) \leq \veps + \sum_{n=0}^N \mu(S_n) = \veps + \sum_{n=0}^N \lim_k \mu_k(S_n) \\ = \veps + \lim_k \mu_k \left( \cup_{n=0}^N S_n \right) \leq \veps + \liminf_k \mu_k(U) \,.
\end{multline*}
Since $\veps > 0$ is arbitrary it follows that $\mu(U) \leq \liminf_k \mu_k(U)$.
\par 
Recalling that $\mu([0,1]) = \mu_k([0,1])$ for all $k \in \N^*$ we infer that for every compact $K \subset [0,1]$,
\begin{multline*}
\mu(K) = \mu([0,1]) - \mu([0,1]\setminus K) \geq \mu_k([0,1]) - \liminf_k \mu_k([0,1]\setminus K) = \limsup_k \mu_k(K) \,.
\end{multline*}
The conclusion follows from Portmanteau's Theorem.
\end{proof}

\begin{Empty}[The mappings $F$ and $F_j$]
\label{number.4}
We associate with $\mu$ its distribution function
\begin{equation*}
f : [0,1] \to [0,1] : t \mapsto \mu([0,t])
\end{equation*}
and we observe that $f$ is continuous (because $\mu$ is diffuse) and nondecreasing.
We also define
\begin{equation*}
F : [0,1] \to \R^2 : t \mapsto (t,f(t)) 
\end{equation*}
and we observe that the set $\Gamma = \rmgraph(f) = F([0,1])$ is $1$-rectifiable and $\calH^1(\Gamma) < \infty$.
This most easily follows from the <<bow-tie lemma>> (see e.g. \cite[4.8.3]{DEP.05c} applied with $n=m+1=2$, $S=\Gamma$, $r=3$, $\sigma=\sin(\pi/4)$ and $W = \rmspan\{e_1+e_2\}$).
\par 
We will approximate $f$ by the functions
\begin{equation*}
f_j : [0,1] \to [0,1] : t \mapsto \mu_j([0,t])
\end{equation*}
which are nondrecreasing and Lipschitz.
Given $t \in [0,1]$ we notice that $\rmBdry [0,t]=\{0,t\}$ is $\mu$-null, whence
\begin{equation*}
f(t) = \mu([0,t]) = \lim_j \mu_j([0,t]) = f_j(t) \,,
\end{equation*}
according to \ref{number.3} and \cite[\S 1.9 Theorem 1]{EVANS.GARIEPY}.
Thus the sequence $(f_j)_{j \in \N^*}$ converges pointwise to $f$.
\par 
We next record that each $f_j$ is differentiable $\calL^1$ almost everywhere.
In fact upon defining
\begin{equation*}
\sigma_j = \frac{1}{2^j \ell_j}
\end{equation*}
one has
\begin{equation*}
f_j'(t) = \begin{cases}
0 & \text{if } t \not \in \cup \calS_j \\
\sigma_j & \text{if } t \in \rmInt \cup \calS_j \,.
\end{cases}
\end{equation*}
\par 
We finally define
\begin{equation*}
F_j : [0,1] \to \R^2 : t \mapsto (t,f_j(t))
\end{equation*}
and related to the Jacobian of $F_j$ we define
\begin{equation*}
\bc_j = 2^j \ell_j \sqrt{1 + \sigma_j^2} = \sigma_j^{-1}  \sqrt{1 + \sigma_j^2} \,.
\end{equation*}
Since $\calL^1(C) = 0$ we infer that $\limsup_j \sigma_j = \infty$ (for otherwise $f$ would be Lipschitz) and in turn
\begin{equation*}
\lim_j \bc_j = \lim_j \sigma_j^{-1}\sqrt{1+\sigma_j^2} = 1 \,.
\end{equation*}
\end{Empty}

\begin{Lemma}
\label{number.5}
For every $j \in \N^*$ and every Borel set $B \subset [0,1]$ one has 
\begin{equation*}
\calH^1(F_j(B)) \geq \bc_j \mu_j(B) \,.
\end{equation*}
\end{Lemma}

\begin{proof}
Let $B \subset [0,1]$ be Borel and define $B' = B \cap (\cup \calS_j)$. 
It follows from the definition of $\mu_j$ that
\begin{equation*}
\mu_j(B) = \mu_j(B') = \sigma_j \calL^1(B') \,.
\end{equation*}
Recalling \ref{number.4} it follows from the <<area formula>> in this simple case (see e.g. \cite[\S 3.3 Theorem 1]{EVANS.GARIEPY} for the general case)
\begin{equation*}
\calH^1(F_j(B)) \geq \calH^1(F_j(B')) = \int_{B'} \sqrt{1 + f_j'(t)^2}d\calL^1(t) = \sqrt{1+\sigma_j^2} \calL^1(B') \,.
\end{equation*}
\end{proof}

\begin{Lemma}
\label{number.6}
Let $S \subset [0,1]$ be any set. It follows that
\begin{equation}
\calH^1(F(S)) \geq \frac{\bar{\mu}(S)}{\sqrt{2}} \,.
\end{equation}
\end{Lemma}

\begin{proof}
We start with the case when $S = [a,b] \subset [0,1]$ is a closed interval. 
Since $F(S) \subset F([0,1])$ is $1$-rectifiable (and compact) the following <<integral geometric inequality>> follows for instance from \cite[3.2.27]{GMT} ($\pi_1$ and $\pi_2$ denote resp. the projection from $\R^2$ onto its first and second axis):
\begin{equation*}
\calH^1(F(S)) \geq \sqrt{a_1^2 + a_2^2}
\end{equation*}
where 
\begin{equation*}
a_1  = \int_\R \rmcard ( F(S) \cap \pi_1^{-1}\{x\} ) d\calL^1(x) 
 = \calL^1(S) 
 = b-a \,,
\end{equation*}
and 
\begin{equation*}
a_2  = \int_\R \rmcard ( F(S) \cap \pi_2^{-1}\{y\} ) d\calL^1(y) 
 = \calL^1(f(S)) 
 = f(b)-f(a) \,.
\end{equation*}
Similarly the other inequality from \cite[3.2.27]{GMT} applies to $F_j(S)$:
\begin{equation*}
a_{1,j} + a_{2,j} \geq \calH^1(F_j(S))
\end{equation*}
where 
\begin{equation*}
a_{1,j}  = \int_\R \rmcard ( F_j(S) \cap \pi_1^{-1}\{x\} ) d\calL^1(x) 
 = \calL^1(S) 
 = b-a \,,
\end{equation*}
and 
\begin{equation*}
a_{2,j}  = \int_\R \rmcard ( F_j(S) \cap \pi_2^{-1}\{y\} ) d\calL^1(y) 
 = \calL^1(f_j(S)) 
 = f_j(b)-f_j(a) \,.
\end{equation*}
Accordingly,
\begin{equation*}
\begin{split}
\calH^1(F(S)) & \geq \sqrt{ (b-a)^2 + (f(b)-f(a))^2 } \\
& = \lim_j \sqrt{ (b-a)^2 + (f_j(b)-f_j(a))^2 } \quad\quad \text{(by \ref{number.4})} \\
& \geq \frac{1}{\sqrt{2}} \lim_j \left( (b-a) + (f_j(b)-f_j(a)) \right) \\
& \geq \frac{1}{\sqrt{2}} \lim_j \left( a_{1,j} + a_{2,j} \right) \\
& \geq \frac{1}{\sqrt{2}} \limsup_j \calH^1(F_j(S)) \\
& \geq \frac{1}{\sqrt{2}} \limsup_j \bc_j \mu_j(S) \qquad\qquad\qquad\qquad \text{(by \ref{number.5})}  \\
& = \frac{\mu(S)}{\sqrt{2}}  \qquad\qquad \text{(according to \ref{number.3} since $\mu(\rmBdry S)=0$)}
\end{split}
\end{equation*}
This completes the proof in case $S$ is a closed interval.
\par 
We now turn to the case when $S \subset [0,1]$ is Borel.
To this end we define $\nu(B) = \sqrt{2}\calH^1(F(B))$, $B \in \calB([0,1])$.
Since $F(B) \in \calB(\R^2)$ whenever $B \in \calB([0,1])$ (because $F$ is continuous, hence Borel measurable, and injective, see \cite[5.4.5]{SRIVASTAVA}) and since $B_1 \cap B_2 = \emptyset$ implies that $F(B_1) \cap F(B_2) = \emptyset$ it follows that $\nu$ is a measure on $\calB([0,1])$.
Thus $\mu$ and $\nu$ are two finite Borel measures on $[0,1]$ such that $\mu(I) \leq \nu(I)$ whenever $I \subset [0,1]$ is a closed interval.
Since $\nu$ is also clearly diffused we infer that $\mu(I) \leq \nu(I)$ whenever $I = (m2^{-n},(m+1)2^{-n}]$, for some $n \in \N^*$ and $m=0,\ldots,2^n-1$. 
Since each relatively open set $U \subset (0,1]$ is the union of a disjointed sequence of such dyadic semi-intervals it follows that $\mu(U) \leq \nu(U)$.
Finally the outer regularity of $\nu$ yields $\mu(B) \leq \nu(B)$ for all $B \in \calB([0,1])$.
\par 
We come to the case when $S \subset [0,1]$ is arbitrary.
We choose a Borel set $B_1 \subset [0,1]$ such that $S \subset B_1$ and $\bar{\mu}(S)=\mu(B_1)$, we choose a Borel set $B_2 \subset \R^2$ such that $F(S) \subset B_2$ and $\calH^1(F(S)) = \calH^1(B_2)$, we let $B_3 = F^{-1}(B_2) \subset [0,1]$ which is Borel as well, and finally we define $B = B_1 \cap B_3$. 
Since $F$ is injective and $F(S) \subset B_2$ we see that $S = F^{-1}(F(S)) \subset F^{-1}(B_2) = B_3$, thus $S \subset B_1 \cap B_3 = B$.
Therefore $\bar{\mu}(S) \leq \bar{\mu}(B) = \mu(B) \leq \mu(B_1) = \bar{\mu}(S)$ and we conclude that $\bar{\mu}(S)=\mu(B)$.
Similarly, from $S \subset B \subset B_3$ and the definition of $B_2$ we infer that $F(S) \subset F(B) \subset F(B_3) \subset B_2$ and in turn $\calH^1(F(S)) \leq \calH^1(F(B)) \leq \calH^1(B_2) = \calH^1(F(S))$ so that $\calH^1(F(S)) = \calH^1(F(B))$.
Finally it follows from the previous paragraph that
\begin{equation*}
\calH^1(F(S)) = \calH^1(F(B)) \geq \frac{\mu(B)}{\sqrt{2}} = \frac{\bar{\mu}(S)}{\sqrt{2}} \,.
\end{equation*}
\end{proof}

\begin{Theorem}[$\mathsf{ZFC} + \sfnon(\calN_{\calL^1}) < \sfcov(\calN_{\calL^1})$]
\label{number.9}
Assume that
\begin{enumerate}
\item $C \subset [0,1]$ is a Cantor set such as in \ref{number.1} and $X = C \times [0,2] \subset \R^2$;
\item $\calA$ is a $\sigma$-algebra of subsets of $X$ such that $\calB(X) \subset \calA \subset \calP(X)$;
\item $\calN \subset \calA$ is a $\sigma$-ideal with the following property: 
\begin{enumerate}
\item $\{x\} \in \calN$ for every $x \in X$;
\item For every $A \in \calA$ and every $1$-rectifiable set $M \subset \R^2$ if $A \cap M \in \calN$ then $\calH^1(A \cap M)=0$;
\end{enumerate}
\item $\sfnon(\calN_{\calL^1}) < \sfcov(\calN_{\calL^1})$.
\end{enumerate}
It follows that $(X,\calA,\calN)$ is not localizable.
\end{Theorem}

\begin{proof}
In this proof $e_1,e_2$ denotes the canonical basis of $\R^2$ and $\pi_1, \pi_2$ the canonical projections of $\R^2$ on its first and second axis respectively.
The result will be obtained as a consequence of \ref{abstract.theorem} applied to $(X,\calA,\calN)$ as in the statement, $(S,\calB(S),\sigma) = (C,\calB(C),\mu)$, $(T,\calB(T),\tau) = ([0,1],\calB([0,1]), \calL^1)$,
\begin{equation*}
V_s = \{s\} \times [0,2] \in \calB(X) \subset \calA \,,
\end{equation*}
$s \in C$, and
\begin{equation*}
H_t = (\Gamma + t.e_2) \cap X \in \calB(X) \subset \calA \,,
\end{equation*}
$t \in [0,1]$, where $\Gamma = F([0,1])$.
\par 
We now check that condition (1) of \ref{abstract.theorem} is satisfied.
Let $s \in C$ and $t \in [0,1]$.
Since $H_t$ is contained in the graph of a function and $V_s$ is contained in a vertical line, 
$V_s \cap H_t$ is either empty or a singleton, therefore a member of $\calN$ according to our current hypothesis (3)(a). 
It is easy to see that $p_{s,t} = (s,f(s)+t) \in V_s \cap H_t$, so that $V_s \cap H_t \neq \emptyset$.
\par 
We next verify that condition (2)(a) of \ref{abstract.theorem} is satisfied.
Fix $s \in C$ and $Z \in \calA$ such that $V_s \cap Z \in \calN$.
Observe that
\begin{multline*}
[0,2] \cap \{ t : H_t \cap V_s \cap Z \neq \emptyset \} = [0,1] \cap \{ t : p_{s,t} \in V_s \cap Z \} \\ = [0,2] \cap \{ t : t \in \pi_2( V_s \cap Z ) - f(s) \} \,,
\end{multline*}
and therefore
\begin{equation*}
\calL^1( [0,2] \cap \{ t : H_t \cap V_s \cap Z \neq \emptyset \}) \leq \calH^1( V_s \cap Z)=0
\end{equation*}
where the last equality follows from our assumption (3)(b) because $V_s$ is 1 rectifiable.
\par 
Finally we ought to show that condition (2)(b) of \ref{abstract.theorem} is satisfied.
Let $t \in [0,1]$ and $Z \in \calA$ be such that $H_t \cap Z \in \calN$.
Observe that
\begin{equation*}
C \cap \{ s : V_s \cap H_t \cap Z \neq \emptyset \} = C \cap \{ s : p_{s,t} \in H_t \cap Z \} = \pi_1(H_t \cap Z) \,.
\end{equation*}
Since $H_t \cap Z = (\Gamma + t.e_2) \cap Z$ and $\Gamma + t.e_2$ is 1 rectifiable, our hypothesis (3)(b) implies that $\calH^1(H_t \cap Z)=0$.
Abbreviating $E = \pi_1(H_t \cap Z) \subset C$ it ensues from \ref{number.6} that
\begin{equation*}
0 = \calH^1(H_t \cap Z) = \calH^1(F(E)+t.e_2) = \calH^1(F(E)) \geq \frac{\bar{\mu}(E)}{\sqrt{2}}  \,,
\end{equation*}
and the proof is complete.
\end{proof}

\begin{Corollary}
\label{main.result}
Let $C \subset [0,1]$ be a Cantor set as in \ref{number.1} and $X = C \times [0,2]$. 
It follows that $(X,\calA_{\calH^1},\calH^1)$ is undecidably semilocalizable.
\end{Corollary}

\begin{proof}
It is consistently semilocalizable according to \ref{CH.implies.ad} and it is consistently not semilocalizable according to \ref{number.9} applied with $\calA = \calA_{\calH^1}$ and $\calN=\calN_{\calH^1}\left[ \calA_{\calH^1}^f\right]$.
\end{proof}

\section{Concluding Remarks and Open Questions}

\begin{Empty}
\label{CR.1}
One may apply \ref{number.9} to other $\sigma$-ideals than $\calN_{\calH^1}$. 
For example let
\begin{equation*}
\calN_{pu} = \calP(\R^2) \cap \{ S : S \text{ is purely $(\calH^1,1)$ unrectifiable} \} \,.
\end{equation*}
Recall that a set $S \subset \R^2$ is called {\em purely $(\calH^1,1)$ unrectifiable} whenever $\calH^1(S \cap M) = 0$ for every 1 rectifiable $M \subset \R^2$.
It then follows from \ref{number.9} that for any $\sigma$-algebra $\calB(X) \subset \calA \subset \calP(X)$ the measurable space with negligibles $(X,\calA,\calA \cap \calN_{pu})$ is consistently not localizable.
\end{Empty}

\begin{Empty}
\label{CR.2}
We turn back to \ref{number.9} applied with $\calN_{\calH^1}$. 
It follows that $(X,\calB(X),\calB(X) \cap \calN_{\calH^1})$ is consistently not semilocalizable.
It further follows in $\mathsf{ZFC}$ from \ref{56} that $(X,\calB(X),\calB(X) \cap \calN_{\calH^1})$ is not almost decomposable.
\begin{enumerate}
\item[(Q1)] {\it I do not know whether, in $\mathsf{ZFC}$, $(X,\calB(X),\calB(X) \cap \calN_{\calH^1})$ is not semilocalizable.}
\end{enumerate}
Notice that, under $\mathsf{CH}$, semilocalizability does not follow from \ref{CH.implies.ad}.
A more general version of this question is the following.
\begin{enumerate}
\item[(Q2)] {\it I do not know whether in \ref{55} the word <<almost decomposable>> might be replaced by the word <<semilocalizable>> without affecting the validity of the statement.}
\end{enumerate}
Notice that \ref{mcshane} does not seem to apply in this situation, for the following reason.
If $(X,\calB(X),\mu)$ is such that $X$ is Polish and $\mu$ is semifinite, then I do not see a reason that the completion of $(X,\calB(X),\mu)$ be locally determined. 
\par 
Another consequence of \ref{number.9} is that there does not exist, in $\mathsf{ZFC}$, a <<localizable version>> of $(X,\calB(X),\calB(X) \cap \calN_{\calH^1})$ obtained by simply <<enlarging>> the given $\sigma$-algebra $\calB(X)$ to another one $\calB(X) \subset \calA \subset \calP(X)$.
Instead it seems necessary to enlarge $X$ first. 
%
\begin{comment}
As stated in \ref{category}(Q4) it would be interesting to investigate the following.
\begin{enumerate}
\item[(Q7)] {\it Assuming one has defined a specific left adjoint functor to the forgetful functor $\sfLOC \to \sfMSN$, describe its effect on $(X,\calB(X),\calB(X) \cap \calN_{\calH^1})$. It would be particularly interesting to give a geometric interpretation of the process.}
\end{enumerate}
\end{comment}
\end{Empty}

\begin{Empty}
\label{CR.3}
Here we ask whether the behavior exhibited by the specific set $X$ of section \ref{example} is shared by other compact subsets of $\R^2$ of Hausdorff dimension 1 but non $\sigma$-finite $\calH^1$ measure.
\begin{enumerate}
\item[(Q3)] {\it Let $X \subset \R^2$ be a compact set of Hausdorff dimension 1 and such that $\calH^1(X \cap U) = \infty$ for every open set $U \subset \R^2$ with $X \cap U \neq \emptyset$. Is $(X,\calA_{\calH^1},\calH^1)$ undecidably semilocalizable?}
\end{enumerate}
\end{Empty}

\begin{Empty}
\label{CR.4}
In regard to (Q3), of particular interest would be an example of such $X$ which is purely $(\calH^1,1)$ unrectifiable.
Let us for instance consider the following set $X$, using the notations of \ref{pue.1}.
Choosing $\lambda_k = k.4^{-k}$ one checks that $X$ has Hausdorff dimension 1 and $\calH^1(X \cap U) = \infty$ whenever $U \subset \R^2$ is open and $X \cap U \neq \emptyset$.
Indeed $\calH^d(X) = 0$ when $1 < d$, is a consequence of the definition of Hausdorff measure, and if $S \in \calX_k$ for some $k \in \N^*$ then $\calH^1(X \cap S) = \infty$ according to \cite[2.10.27]{GMT} because $X \cap S = K \times K$ for some $K \subset [0,1]$ with $\calH^\frac{1}{2}(K)=\infty$ according to \cite[2.10.28]{GMT}.
Also observe, as in \ref{pue.1}(1) that $X$ is purely $(\calH^1,1)$ unrectifiable.        
\begin{enumerate}
\item[(Q4)] {\it With the set $X$ described here, is $(X,\calA_{\calH^1},\calH^1)$ undecidably semilocalizable?}
\end{enumerate}
Viewing $X$ as a product as in section \ref{section.pue} does not seem to be immediately helpful since it gives information about a Hausdorff measure essentially of dimension $1/2$.
Trying to use the graphs of distribution functions as in section \ref{example} is no more successful since these graphs are rectifiable and $X$ is purely unrectifiable ; their intersection will always be $\calH^1$ null.
One may also attempt to produce families $V_s$ and $H_t$ needed in \ref{abstract.theorem} as random Cantor subsets of $X$: the $V_s$ using more often a specific set of three subsquares at each generation, and the $H_t$ using more often a distinct specific set of three subsquares at each generation.
However random sets constructed this way tend to intersect too often, making it hard to guarantee condition (2) of \ref{abstract.theorem}.
\end{Empty}

\begin{Empty}
\label{CR.5}
In the notation of section \ref{section.pue} and with the same restriction on $d$ as in the proof of \ref{pue.8},
\begin{enumerate}
\item[(Q5)] {\it I do not know whether the measurable spaces with negligibles $(X_d,\calA_{\calH^d},\calN_{\calH^d})$ and $\left( [0,1],\calA_{\calH^{1/2}}, \calN_{\calH^{1/2}} \right)$ are isomorphic in the category $\sfMSN$.
}
\end{enumerate}
This boils down to deciding whether the $\sigma$-algebra $\calA$ defined in the proof of \ref{pue.8} coincides with the $\sigma$-algebra $\calA_{\calH^d}$.
The fact the answer to this question is not known turned out to be no obstacle thanks to the freedom allowed in \ref{abstract.theorem} regarding the $\sigma$-algebra $\calA$.
\end{Empty}

\begin{Empty}
\label{CR.6}
Our last question here concerns \ref{pue.8}. 
The proof, based on \ref{abstract.theorem} seems to require a product structure that forces the dimension to be $1/2$.
\begin{enumerate}
\item[(Q6)] {\it 
Let $0 < d < 1$ and $d \neq \frac{1}{2}$.
Is the measure space $\left( [0,1], \calA_{\calH^d} , \calH^d \right)$ consistently not semilocalizable?
}
\end{enumerate}
\end{Empty}

%=======================
% BIBLIOGRAPHY AND INDEX
%=======================

\bibliographystyle{amsplain}
\bibliography{/home/thierry/Documents/LaTeX/Bibliography/thdp}

%\printindex

\end{document}